\documentclass[aps,prx,reprint,superscriptaddress]{revtex4-2}
\usepackage{graphicx}
\usepackage{dcolumn}
\usepackage{bm}
\usepackage{amssymb,amstext,amsmath}
\usepackage{xr}
\usepackage{color}
	 \definecolor{darkred}{rgb}{0.75,0,0}
	 \definecolor{darkgreen}{rgb}{0,0.5,0}
	 \definecolor{darkblue}{rgb}{0,0,0.75}
  	 \definecolor{darkorange}{rgb}{1,0.9,0.1}
	 \definecolor{dark}{rgb}{0,0,0}

\newtheorem{theorem}{Theorem}

\newtheorem{definition}{Definition}
\newtheorem{remark}{Remark}
\newtheorem{proof}{Proof}
\newtheorem{lemma}{Lemma}
\newtheorem{corollary}{Corollary}

\newtheorem{example}{Example}
\begin{document}

\preprint{APS/123-QED}

\title{Convergence of time-delayed opinion dynamics with complex interaction types}

\author{Lingling Yao}
\affiliation{%
Center for Systems and Control, College of Engineering, Peking University, Beijing 100871, China}
\author{Aming Li}
\thanks{Corresponding author: amingli@pku.edu.cn}
\affiliation{%
Center for Systems and Control, College of Engineering, Peking University, Beijing 100871, China}
\affiliation{
Center for Multi-Agent Research, Institute for Artificial Intelligence, Peking University, Beijing 100871, China}




\date{\today}

\begin{abstract}
In opinion dynamics, time delays in agent-to-agent interactions are ubiquitous, which can substantially disrupt the dynamical processes rooted in agents' opinion exchange, decision-making, and feedback mechanisms. However, a thorough comprehension of quantitative impacts of time delays on the opinion evolution, considering diverse interaction types and system dynamics, remains absent. In this paper, we conduct a systematic investigation into the convergence and the associated rate of time-delayed opinion dynamics with diverse interaction types in both discrete-time and continuous-time systems. For the discrete-time system, we commence by establishing sufficient conditions for its convergence on arbitrary signed interaction networks. These conditions show that the convergence is determined solely by the topology of the interaction network and remains impervious to the magnitude of the time delay. Subsequently, we examine the influence of random and other interaction types on the convergence rate and discover that time delays tend to decelerate this rate. Regarding the continuous-time system, we derive the feasible domain of the delay that ensures the convergence of opinion dynamics, revealing that, unlike the discrete-time scenarios, large time delays can instigate the divergence of opinions. Specifically, we prove that for both random and other interaction types, small delays can expedite the convergence of continuous-time system. Finally, we present simulation examples to demonstrate the effectiveness and robustness of our research findings.
\end{abstract}

\maketitle

\section{Introduction}
In the field of opinion dynamics, researchers primarily investigate the phenomena of how diverse perspectives are formed and evolve within populations regarding specific events, influenced by individual decision-making processes and external information. Over the past decades, the analysis of opinion dynamics has received considerable attention from social scientists, physicists, mathematicians, and researchers in the emerging interdisciplinary field of complex systems, due to its wide applications in the stock market, institutional change, election forecasting, task allocation, and engineering systems \cite{Etes15,Pros17,Pros172,Yang23,yao22}.

A model that strikes a balance between the conflicting demands of capturing the inherent complexity of social dynamics and maintaining mathematical tractability is fundamental for the study of opinion dynamics in social networks. The DeGroot model, recognized as one of the earliest and most influential models in the field of opinion dynamics, demonstrates the fundamental human cognitive ability to take convex combinations when integrating relevant information \cite{De:74}. To more accurately capture the complexities of the real social world, a variety of extended models have been proposed, such as the Friedkin-Johnsen model \cite{Fri99,Tia18}, the bounded confidence model \cite{Heg02,Wedin}, the multidimensional model \cite{Fri16}, the expressed and private model \cite{Ye19}, the appraisal dynamics model \cite{Tia22,liu24}, and the Altafini model \cite{Al13}.

 In the real world, due to the limited public resources and the universality of social contradictions, it is inevitable that there are antagonistic or competitive relationships in interaction networks \cite{Al13,Liu19,shi19,Suo,pros16,Liu17,GUAN2023110694,pan}. The model with both friendly and antagonistic interactions called the continuous-time Altafini model was first proposed in \cite{Al13}. In the Altafini model, negative edges are introduced to represent antagonistic interactions, and positive edges correspond to friendly interactions between individuals. Based on the assumption that the interaction network is digon sign-symmetric and strongly connected, the system achieves bipartite consensus when the network is structurally balanced but reaches stability (this phenomenon is also called neutralization \cite{Liu19}) when the network is structurally unbalanced. Following this pioneering work, some scholars have relaxed the requirements for topology and analyzed a broader range of consensus behaviors such as interval bipartite consensus \cite{shi19,Suo}. Meanwhile, the continuous-time Altafini model and its discrete-time version in time-varying signed graphs are further studied in \cite{pros16,Liu17}. For the discrete-time Altafini model with time-varying signed digraph, under the assumption of repeatedly jointly strong connectivity and the existence of self-loop for each node, the researchers establish some necessary and sufficient conditions for exponential convergence \cite{Liu17}. At present, the Altafini model and its extensions have been extensively investigated, and some important results have been derived, including those related to the bounded confidence model and the Friedkin-Johnsen model \cite{Su23,AL18,Carmela24,Gio24}. To simplify the analysis of the discrete-time case, existing literature often assumes that each node has a self-loop, meaning that an individuals' current opinions always directly affect their opinions at the next step, which is interpreted as each individual being self-confident. However, in reality, this condition is quite restrictive, as there exist individuals who lack self-confidence. The necessity of the assumption is unclear, since the system's convergence issue remains unresolved without self-loops.

Time delay is often inevitable in biology, ecology, chemistry,
physics, and numerous engineering disciplines. Moreover, many practical systems such as sampled-data systems and network control systems can be modeled as
time-delayed systems, where large amounts of digital data are transferred and processed. Therefore, it is of importance to study the convergence problem for time-delayed systems from both theoretical and practical points of view. Different from the delay-free systems, the delayed systems may belong to the class of functional differential equations which are infinite dimensional, rather than ordinary differential equations. The methods applicable to finite-dimensional systems are not always suitable for infinite-dimensional systems, and as a result, analytical solutions are often elusive. Up to now, some approaches have been proposed to address the convergence of time-delay systems, such as the Nyquist stability criteria, the Lyapunov-Krasovskii function method, and the linear matrix inequality technique \cite{Olf04,XI08,Had14,Vu10,Zeng15,Cep15,BL08,MA23,xiao,yangy23}. The effect of time delays on agents' consensus behavior was first analyzed for the continuous-time model in \cite{Olf04}, and an explicit formula of the admissible delay bound under the undirected interaction networks was established. Subsequently, there have been extensive studies and significant research findings on consensus problems of first-order \cite{XI08,Had14}, second-order, general linear \cite{Vu10,Zeng15}, high-order \cite{Cep15}, and non-linear systems \cite{BL08}. It is worth noting that the intentionally induced delays can have a positive effect on the
consensus of second-order systems \cite{MA23}. However, further investigation is required to elucidate the positive effects on the convergence of opinion dynamics associated with the occurrence of delay.

The convergence rate determines  the swiftness with which a system transitions from its initial state to a state of consensus, providing an important quantitative indicator for decision-making and engineering practice \cite{Ghader}. For instance, in the decision-making process, an emergent group opinion or consensus often has to be made within a short finite time. In the engineering field, it is desirable for disaster response and relief operations to rapidly attain the vicinity of the drop-off point. Therefore, some researchers have shown particular interest in it and made some remarkable progress. The existing literature mainly addresses accelerated convergence from two perspectives. One perspective is identifying the optimal network topology, within structural constraints, to maximize the convergence rate for a given protocol (e.g., by optimizing the weight matrix \cite{A17}). For directed acyclic graphs, enhancing the convergence rate can be achieved by adding certain edges, as suggested in \cite{Zhang17}. The other perspective is seeking the optimal protocol to improve the convergence rate for a given network topology (e.g., utilizing finite-time control \cite{Wan10,Xia14,ANGULO2012606}, employing graph signal processing \cite{Yi20}, and introducing individual memory \cite{Yi23, REN2022}). Despite the extensive research and abundant outcomes garnered by systems with time delays, the convergence rate of opinion dynamics with time delay on signed networks remains unclear.

This paper conducts an in-depth study on the convergence and convergence rates of delayed systems with complex interaction types by addressing the following key problems:
\\
(I). Is the presence of self-loops in the interaction network a necessary condition to ensure the convergence of discrete-time systems? If not, how to provide a class of interaction networks without self-loops but can still ensure the convergence of discrete-time systems?
\\
(II). Does the existence of time delay lead to the divergence of opinions? If so, what are the feasible regions of the delays to ensure the convergence of delayed systems with complex interaction types?
\\
(III). What are the effects of the delays on the convergence rate of delayed systems with complex interaction types?

Motivated by previous discussions, we aim to develop a framework for analyzing the convergence and convergence rate of delayed opinion dynamics with complex interaction types for both discrete-time and continuous-time systems. The main contributions of this article are summarized as follows.
\begin{itemize}
  \item For the discrete-time system, from requiring self-loops in each closed strongly connected component to allowing the absence of self-loops, we have demonstrated that it can still converge under certain conditions. Specifically, convergence can still be achieved for both random mixture interaction and other interaction scenarios without self-loop, indicating that the presence of self-loop is not a necessary condition for the convergence of discrete-time systems (see Theorems \ref{the1}-\ref{the3}).
  \item For the discrete-time system, rigorous theoretical analysis of the impact of time delays on its convergence rate has been conducted for random mixture interaction and complex mixed interaction scenarios, which can be regarded as an extension and expansion of the research of delay-free systems (see Theorem \ref{the4}).
  \item For the continuous-time system, the feasible regions of time delays to ensure its convergence for random mixture interaction and complex mixed interaction scenarios are provided, establishing a direct connection between the convergence of delayed systems and some key factors such as population size, network connectivity, and various interaction types (see Theorems \ref{the8}-\ref{the9}, and Corollary \ref{cor1}).
  \item For the continuous-time system, we prove that there always exists a small delay that can accelerate its convergence for both random mixture interaction and complex mixture scenarios (see Theorem \ref{the12}).
\end{itemize}
~~The rest of this paper is organized as follows. Section I offers some useful preliminary knowledge of signed graphs and the construction of networks with diverse types considered in this paper. In Section II, we introduce the opinion dynamics model and some useful lemmas. In Section III, we present our main theoretical results. The theoretical results are verified by numerical simulations in Section IV. Finally, Section V concludes this paper.

The notations and abbreviations used in this paper are listed in Table \ref{tab1}.
\begin{table}[ht]\label{tab1}
\centering{Table 1: notations}
\begin{tabular}{p{2.3cm}p{5.2cm}}
\hline Symbols & Definitions \\
\hline
$\mathbb{C}$ & set of complex numbers \\
$\mathbb{C}_{-}$ & set of complex numbers with negative real
part \\
$\mathbb{R}$ & set of real numbers \\
$\mathbb{N}$ & set of negative integers\\
$\mathbb{Z}$ & set of integers\\
$\mathbb{R}^{m\times n}$ & set of $m\times n$ real matrices \\
$I$ & any identity matrix with proper dimension \\
$\mathrm{diag}\{a_{1}, \cdots, a_{n}\}$ & diagonal matrix with diagonal elements $a_{1}, \cdots, a_{n}$\\
$A^{\top}$ & the transpose of matrix $A$\\
$A^{-1}$ & the inverse of matrix $A$\\
$|A|$ & nonnegative matrix in which each element $|A|_{ij}$ equals $|A_{ij}|$\\
$|A|_{i}$ & sum of the elements in the $i$th row of matrix $|A|$ \\
$\lambda_{i}(A)$ & the $i$th eigenvalue of matrix $A$ \\
$\mathrm{Re}(\lambda_{i}(A))$ & real part of $\lambda_{i}(A)$ \\
$\mathrm{Im}(\lambda_{i}(A))$ & imaginary part of $\lambda_{i}(A)$ \\
$\rho(A)$ & spectral radius of matrix $A$ \\
$[n]$ & set of $1,2,\cdots, n$ \\
$\mathrm{sgn}(x)$ & sign function with respect to $x$\\
\hline
\end{tabular}
\end{table}
\section{Preliminaries and model formulation}\label{sect1}
\subsection{Signed networks}
 Let $\mathcal{G}(W)=(\mathcal{V},\mathcal{E},W)$ be a (weighted) signed graph of $n$ nodes, with the nodes set $\mathcal{V}\subseteq\{v_{1},\dots, v_{n}\}$, the edges set $\mathcal{E}=\mathcal{V}\times \mathcal{V}$, and the signed weighted matrix $W=[w_{ij}]$, $i, j\in[n]$. An edge $(v_{i}, v_{j})\in\mathcal{E}$ means that node $j$ can get information from node $i$. $w_{ij}\neq0\Leftrightarrow\left(v_j, v_i\right) \in \mathcal{E}$. Moreover, the weight $w_{ij}>0$ $(w_{ij}<0)$ denotes the trust (mistrust) that individual $j$ has on individual $i$.  A digraph $\mathcal{G}(W)$ is a signed network if there exist at least one negative edge in $\mathcal{G}(W)$. A signed digraph $\mathcal{G}(W)$ is structurally balanced if there exists a bipartition $\left\{\mathcal{V}_1, \mathcal{V}_2\right\}$ of the vertices, where $\mathcal{V}_1 \cup \mathcal{V}_2=\mathcal{V}$ and $\mathcal{V}_1 \cap \mathcal{V}_2=\varnothing$, such that $w_{i j} \geq 0$ for $\forall v_i, v_j \in \mathcal{V}_l$ $(l \in\{1,2\})$ and $w_{i j} \leq 0$ for $\forall v_i \in \mathcal{V}_l, v_j \in$ $\mathcal{V}_q, l \neq q$ $(l, q \in\{1,2\})$; and $\mathcal{G}$ is structurally unbalanced otherwise \cite{Al13}.
 A directed walk of $\mathcal{G}(W)$ is a sequence of edges $\{(v_{1},v_{2}),(v_{2},v_{3}),\cdots,(v_{m-1},v_{m})\}\subseteq\mathcal{E}$.
Moreover, a directed walk in which all vertices are distinct is said to be a directed path. In contrast, if all the vertices in a directed walk are distinct except for the same starting and ending vertices, it is called a directed cycle. A signed directed cycle with an even/odd number of edges having negative weights is called a positive/negative directed cycle \cite{Al13}. If all nodes in this sequence are distinct, then it is called a path. Agent $j$ is a reachable node of agent $i$ if there exists a path from agent $i$ to agent $j$. A graph is periodic if it has at least one cycle and the length of any cycle is divided by some integer $h>1$. Otherwise, a graph is called aperiodic. A strongly connected subgraph $\mathcal{G}^{'}$ of digraph $\mathcal{G}$ is called a strongly connected component (SCC) if it is not contained by any larger strongly connected subgraph. A SCC without incoming arcs from other SCCs is called a closed SCC (CSCC); otherwise, it is called an open SCC (OSCC). For a digraph $\mathcal{G},$ $\left\{\mathcal{V}_1, \ldots, \mathcal{V}_n\right\}$ is called a partition of $\mathcal{G}$ if $\mathcal{V}_i \neq$ $\varnothing, \bigcup_{i=1}^n \mathcal{V}_i=\mathcal{V}$, and $\mathcal{V}_i \cap \mathcal{V}_j=\varnothing$ for $i \neq j$, where $\mathcal{V}_i$ is the node set corresponding to the subnetwork $\mathcal{G}_i$. Then, the compressed graph $\mathbb{G}$ of $\mathcal{G}$ is defined as follows: (a) The node set of $\mathbb{G}$ consists of $\mathcal{G}_1, \ldots, \mathcal{G}_n$. That is, each $\mathcal{G}_i$ is deemed as a node in $\mathbb{G}$. (b) There is an edge $\left(\mathcal{G}_i, \mathcal{G}_j\right) \in \mathcal{E}(\mathbb{G})$ from node $\mathcal{G}_i$ to node $\mathcal{G}_j$ in $\mathbb{G}$ if and only if there exist nodes $u \in \mathcal{V}_i$ and $v \in \mathcal{V}_j$ such that $(u, v) \in \mathcal{E}(\mathcal{G})$.
\subsection{Opinion dynamics on signed networks}
The discrete-time and continuous-time Altafini models without delays have been extensively studied \cite{Al13,Liu19}. In this paper, we examine these models with delays included and emphasize the significant impact that delays and complex interaction types have on these models.
\subsubsection{Discrete-time opinion dynamics}\label{sub1}
The discrete-time opinion dynamics with delays can be described by
\begin{equation}\label{equ1}x_i(k+1)=w_{i i}x_i(k)+\sum_{j\in\mathcal{N}_{i},j\neq i} w_{i j}x_j(k-\tau_d),\end{equation}
   where $x_i(k)\in[-1,1]$ represents the opinion of individual $i$ at time $k$ and  $\tau_d\in\mathbb{N}$ is the
transmission delay of information from $v_{j}$ to $v_{i}$. $w_{i j}\in[-1,1]$ reflects the interpersonal influence that agent $j$ has on agent $i$ and $w_{i i}\geq0$ denotes the self-loop on node $i$, capturing the self-confidence level of individual $i$. Moreover, $\sum\limits_{j=1}^{n}|w_{ij}|=1$ for $\forall i\in[n]$. $x_i(k)>0$, $x_i(k)<0$ and $x_i(k)=0$ represents the support, rejection, and neutrality of individual $i$, respectively. In addition, when $\tau_{d}>k$, $x_i(k+1)=x_i(k)=\cdots=x_i(0)$.

 Let $X(k+1)=[x_1(k+1),\cdots,x_n(k+1)]^{\top}$, then the system (\ref{equ1}) can be rewritten as
  \begin{equation}\label{equ2}X(k+1)=\hat{W}X(k)+\tilde{W}X(k-\tau_d),\end{equation}
   where matrix $W=\tilde{W}+\hat{W}$ is a signed stochastic matrix, i.e., $w_{i j}\in[-1,1]$ and $\sum\limits_{j=1}^{n}|w_{i j}|=1, \forall i\in[n]$.  $\hat{W}=\mathrm{diag}\{w_{11},\cdots,w_{nn}\}$ and  $\Xi=\mathrm{diag}\{\xi_{1},\cdots,\xi_{n}\}$. In order to facilitate our analysis, define $Y(k+1)=[x(k+\tau_{d})^{\top},x(k+\tau_{d}-1)^{\top},\cdots,x(k)^{\top}]^{\top}$, then system (\ref{equ2}) can be equivalently converted into \begin{equation}\label{equ3}Y(k+1)=AY(k),\end{equation}
 where \begin{equation}\label{afrom}
 A=[a_{ij}]=\begin{bmatrix}
\hat{W} & 0 & \cdots & 0 & \tilde{W} \\
I & 0 & \cdots & 0 & 0 \\
0 & I & \cdots & 0 & 0 \\
\vdots & \vdots & \ddots & \vdots & \vdots \\
0 & 0 & \cdots & I & 0
\end{bmatrix}\in\mathbb{R}^{n(\tau_{d}+1)\times n(\tau_{d}+1)}\end{equation} and $\sum\limits_{j=1}^{n}|a_{i j}|=1, \forall i\in\{1,2,\cdots,n(\tau_{d}+1)\}$.
\\
\indent
Hence, the convergence of the system (\ref{equ1}) can be equivalently converted into that of the augmented system (\ref{equ3}), which can be regarded as the system with $n(\tau_{d}+1)$ individuals under the signed communication topology $\mathcal{G}(A)$.
\begin{remark}
The signed digraph $\mathcal{G}(A)$ can be regarded as a multi-layer network, where there are $n$ layers and $\tau_{d}+1$ nodes in each layer. The network partition based on $n$ layers of $\mathcal{G}(A)$ can be useful for clarifying the specific relationship concerning the properties of structural balance and topological structure between $\mathcal{G}(W)$ and $\mathcal{G}(A)$ (see Lemmas 1-3 for details).
\end{remark}
Let $V_{i}:= \{i, n+i, 2n+i, \cdots, n\tau_{d}+i\}$, $i \in [n]$. Then, $\{V_{1}, V_{2}, \cdots, V_{n}\}$ is a partition of $\mathcal{G}(A)$, where $\mathcal{G}_{1},\cdots,\mathcal{G}_{n}$ can be deemed as $n$ nodes in the compressed graph $\mathbb{G}$ corresponding to this partition. Furthermore, based on this partition, we give the following lemmas concerning the compressed graph $\mathbb{G}$ and $\mathcal{G}(W)$.
\begin{lemma}\label{lem1}
There exists an edge pointing from node $\mathcal{G}_j$ to node $\mathcal{G}_i$ in the compressed graph $\mathbb{G}$ of $\mathcal{G}(A)$ if and only if there exists an edge pointing from node $j$ to node $i$ in $\mathcal{G}(W)$, $i, j\in[n]$.
\end{lemma}
\begin{proof}
From the form (\ref{afrom}) of matrix $A$, we have $a_{j,n\tau_{d}+i}=w_{ji}$, $i, j\in[n]$. Moreover, according to the partition method of $\mathcal{G}(A)$, this lemma obviously holds.
\end{proof}
\begin{lemma}\label{lem14}
The number of \textup{CSCCs} \textup{(OSCCs)} in $\mathcal{G}(W)$ is identical to that in $\mathcal{G}(A)$.
\begin{proof}
Since this lemma can be easily obtained from Lemma \ref{lem1}, we omit it here.
\end{proof}
\end{lemma}
\begin{lemma}\label{lem4}
$\mathcal{G}(W)$ is structurally unbalanced (structurally balanced) if and only if $\mathcal{G}(A)$ is structurally unbalanced  (structurally balanced).
\begin{proof}
Since in subnetwork $\mathcal{G}_{i}$ of $\mathcal{G}(A)$, all edges are positive, $\forall i\in[n]$, by the definition of structural balance, the structural balance of $\mathcal{G}(A)$ is identical to that of $\mathbb{G}$. In addition, we have $$w_{ji}>0(<0)\Leftrightarrow a_{j,n\tau_{d}+i}=w_{ji}>0(<0), i, j\in[n],$$ which means that there exists a positive (negative) edge pointing from node $j$ to node $i$ in $\mathcal{G}(W)$ if and only if there exists a positive (negative) pointing from node in $V_{j}$ to node $V_{i}$. By the definition of structural balance again, the structural balance of $\mathbb{G}$ is identical to that of $\mathcal{G}(W)$. Hence, we conclude that the structural balance of $\mathcal{G}(A)$ is identical to that of $\mathcal{G}(W)$.
\end{proof}
\end{lemma}
In the sequel, an example is given to illustrate the relationship between $\mathcal{G}(A)$ and $\mathcal{G}(W)$.
\begin{example}
Consider the system (\ref{equ2}) with interaction network $\mathcal{G}(W)$ in Fig. \ref{pic-xiu} (a) and time delay $\tau_d=2$.
From (\ref{afrom}), $\mathcal{G}(A)$ is a multi-layers network with three layers (see Fig. \ref{pic-xiu} (b)). From Fig. \ref{pic-xiu}, it can be observed that there is a CSCC and an OSCC in both $\mathcal{G}(W)$ and $\mathcal{G}(A)$. Moreover, both $\mathcal{G}(W)$ and $\mathcal{G}(A)$ are structurally unbalanced. This confirms the theoretical results presented in Lemmas \ref{lem1}-\ref{lem4}.
 \end{example}
 \begin{figure}
    \centering
    \includegraphics[scale=0.35]{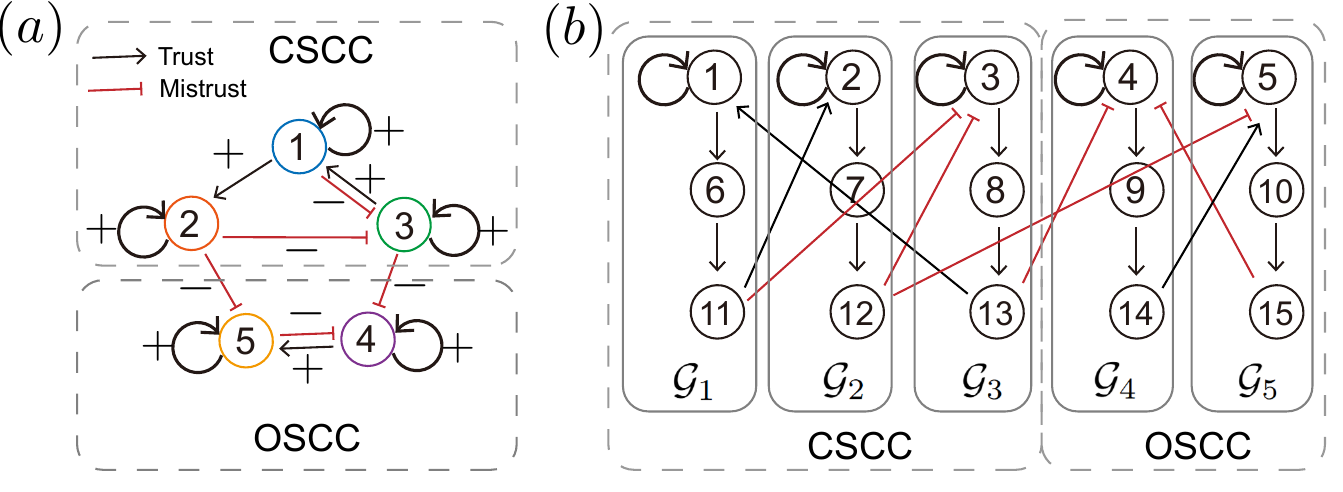}
    \caption{The interaction networks $\mathcal{G}(W)$ and $\mathcal{G}(A)$. Black lines and red lines represent trust and mistrust interactions, respectively.}
    \label{pic-xiu}
 \end{figure}
 \subsubsection{Continuous-time opinion dynamics}\label{sub2}
 The continuous-time Altafini model with delays can be described by
\begin{equation}\label{equ1c}\dot{x}_i(t)=\sum_{j\in\mathcal{N}_{i}} w_{i j}(x_j(t-\tau_c)-\mathrm{sgn}(w_{ij})x_i(t-\tau_c)),\end{equation}
 where $x_i(k)\in[-1,1]$ and $w_{ij}\in[-1,1]$ are defined in Subsection \ref{sub1}. $\tau_c\in\mathbb{R}$ denotes the delays in the continuous-time system (4).

Similarly, by defining $X(t)=[x_1(t),\cdots,x_n(t)]^{\top}$, the system (\ref{equ1c}) can be rewritten as
\begin{equation}\label{equ1c1}\dot{X}(t)=-LX(t-\tau_{c}),\end{equation}
where matrix $L=[l_{ij}]\in\mathbb{R}^{n\times n}$ is defined as
 \begin{align}\label{c3}l_{ij}=\left\{\begin{array}{l}
 w_{ij}, \ \text{if} \ i\neq j,\\
   -\sum\limits_{k, k\neq i}|w_{ik}|, \ \text{otherwise}.\\
   \end{array}\right.\end{align}

 Taking the Laplace transform of both sides of Eq. (\ref{equ1c1}), we have
$$G(s)=(sI+\mathcal{L}(s))^{-1},$$
where $\mathcal{L}(s)=e^{-s\tau_c}L$.
The convergence of system (\ref{equ1c1}) can be determined by checking the roots of the following characteristic equation
\begin{equation}\label{c4}\prod_{k=1}^n\left(z-\alpha_{k}\mathrm{e}^{-z \tau_c}\right)=0,\end{equation}
where the eigenvalues $\alpha_{i}:=\lambda_{i}(-L)$ of matrix $-L$ are ordered according to \begin{equation}\label{equ13}|\text{Re}(\alpha_{1})|\leq|\text{Re}(\alpha_{2})|
\leq\cdots\leq|\operatorname{Re}(\alpha_{n})|.\end{equation}

Next, the following basic lemmas are needed to analyze the convergence of the systems (\ref{equ2}) and (\ref{equ1c1}).
 \begin{lemma}\label{lem810} (See Better Theorem in \cite{Horn}) Let $\lambda$ and $x$ be an eigenpair of matrix $A$ such that $\lambda$ satisfies the following inequalities: $$\left|\lambda-a_{i i}\right| \geq R_{i}^{\prime}=\sum\limits_{j \neq i}\left|a_{i j}\right|, \ \forall i=1, \cdots, n.$$ If digraph $\mathcal{G}(A)$ is strongly connected, then every Gersgorin Circle passes through $\lambda$.
 \end{lemma}
\begin{definition}\label{def0}
For any initial conditions, if the limit $\lim\limits_{k\to\infty}X(k)$ exists, the discrete-time system (\ref{equ2}) converges. Furthermore, the system  (\ref{equ2}) achieves bipartite consensus (stability) if $\forall i\in[n]$, $\lim _{k \rightarrow \infty}\left|x_i(k)\right|=\alpha>0 (=0),$ where $\alpha$ is a constant.
\end{definition}
Since the definitions above for the continuous-time system (\ref{equ1c1}) can be similarly defined, we omit it here.
\section{Main results}\label{sect3}
In this section, some convergence criteria for both discrete-time and continuous-time delayed systems are established. Furthermore, the issue of convergence rate for delayed systems in the presence of complex interaction types will be resolved through rigorous theoretical analyses, revealing the significant impact of interaction type on the convergence rate of delayed systems.
\subsection{Convergence of the discrete-time opinion dynamics}\label{subsec3.1}
Following the augmented system (\ref{equ3}) in Subsection \ref{sub1}, we aim to relax the assumptions  regarding the convergence of the system (\ref{equ2}) presented in some existing literature. Specifically, for the interaction network $\mathcal{G}(W)$, we transition from the condition where every node has one self-loop, to the scenario where some nodes lack self-loops, and finally to the scenario where there are no self-loops at all.
\begin{theorem}\label{the1}
If there is at least one self-loop in each \textup{CSCC} of $\mathcal{G}(W)$, then the discrete-time system (\ref{equ2}) converges.
\begin{proof}
Since $\rho(A)\leq\|A\|_{\infty}=\max\limits_{i\in\{1,\cdots,n(\tau_{d}+1)\}}\sum\limits_{j=1}^{n(\tau_d+1)}|a_{ij}|=1$, $\rho(A)\leq 1$. Without loss of generality, suppose there exist $\alpha$ \textup{CSCCs} and $\beta$ \textup{OSCCs} in $\mathcal{G}(W)$. By Lemmas 1-3, there are $\alpha$ \textup{CSCCs} in $\mathcal{G}(A)$ and each \textup{CSCC} in $\mathcal{G}(A)$ has at least one self-loop.

In the sequel, let matrix $A$ be transformed into the ``canonic'' form as follows:
\begin{equation}\label{equ9}\left[\begin{array}{ccccc}
A_{11} & * & * & \ldots & * \\
 O & A_{22} & * & & * \\
 O &  O & A_{33} & & * \\
\vdots & \vdots& \vdots & \ddots & \vdots \\
 O &  O &  O & \ldots & A_{s s}
\end{array}\right],\end{equation}
where each block matrix $A_{ii}$ is irreducible, and $\text{Eig}(A)$ and $\text{Eig}(A_{ii})$ are the sets of eigenvalues of $A$ and $A_{ii}$, respectively. Moreover, $\text{Eig}(A)=\bigcup_{i=1}^{\alpha} \text{Eig}(A_{ii})$.

When $\mathcal{G}(A_{ii})$ is an \textup{OSCC}, by Lemma \ref{lem810}, $\rho(A_{ii})<1$. When $\mathcal{G}(A_{ii})$ is a \textup{CSCC}, according to the structure balance of $\mathcal{G}(A_{ii})$, the distribution of eigenvalues of $A$ is divided into two cases:
\\
(1). $\mathcal{G}(A_{ii})$ is structurally balanced. Since $A_{ii}$ is irreducible and aperiodic, $A_{ii}$ is primitive. By the definition of primitive matrix, 1 is the unique nonzero eigenvalue of maximum modulus and 1 is an algebraically simple.
\\
(2). $\mathcal{G}(A_{ii})$ is structurally unbalanced. Next, we prove that $\rho(A_{ii})<1$ by contradiction. In this case,  by the property of M-matrix, Laplacian matirx $L=I-A_{ii}$ is positive stable, i.e., $\mathrm{Re}(\lambda_{j}(L))>0$. Suppose $\rho(A_{ii})=1$ and 1 is an eigenvalue of $A$, then $0$ is an eigenvalue of $L$, which contradicts the positive stability for matrix $L$. Then, we need to prove $c+di$ $(d\neq 0)$ satisfying $c^{2}+d^{2}=1$ is not an eigenvalue of $A$. If $c+di$ $(c<1, d\neq 0)$ satisfying $c^{2}+d^{2}=1$ is an eigenvalue of $A$, considering $\exists a_{ii}>0$, we have
\begin{align}
\left|\lambda-a_{i i}\right| & =\sqrt{\left(c-a_{i i}\right)^2+d^2}\nonumber \\
& =\sqrt{1-2 ca_{i i}+a_{i i}^2}\nonumber \\
& >\sqrt{\left(1-a_{i i}\right)^2}\nonumber \\
& =\sum_{j \neq i}\left|a_{i j}\right|, i=1,2, \ldots, n.\label{equ6}
\end{align}
If $c<0$, then $2ca_{i i}\leq 2a_{i i}$; if $c>0$, due to $c^{2}+d^{2}=1$, then $c<1$ and $2ca_{i i}\leq 2a_{i i}$. Therefore, the inequality (\ref{equ6}) holds. By Lemma \ref{lem810}, $c+di$ is not an eigenvalue of $A$, which contradicts our assumption. Thus, $\rho(A_{ii})<1$.

Summarizing the above analyses, we conclude that  system  (\ref{equ2}) converges.
\end{proof}
\end{theorem}
\begin{corollary}\label{cor2}
If $\mathcal{G}(W)$ is strongly connected and there is at least one self-loop in $\mathcal{G}(W)$, then system (\ref{equ2}) converges. Specifically, when $\mathcal{G}(W)$ is structurally balanced (structurally unbalanced),  system (\ref{equ2}) achieves bipartite consensus (stability).
\end{corollary}
\begin{figure*}[ht]
    \centering
    \includegraphics[scale=0.53]{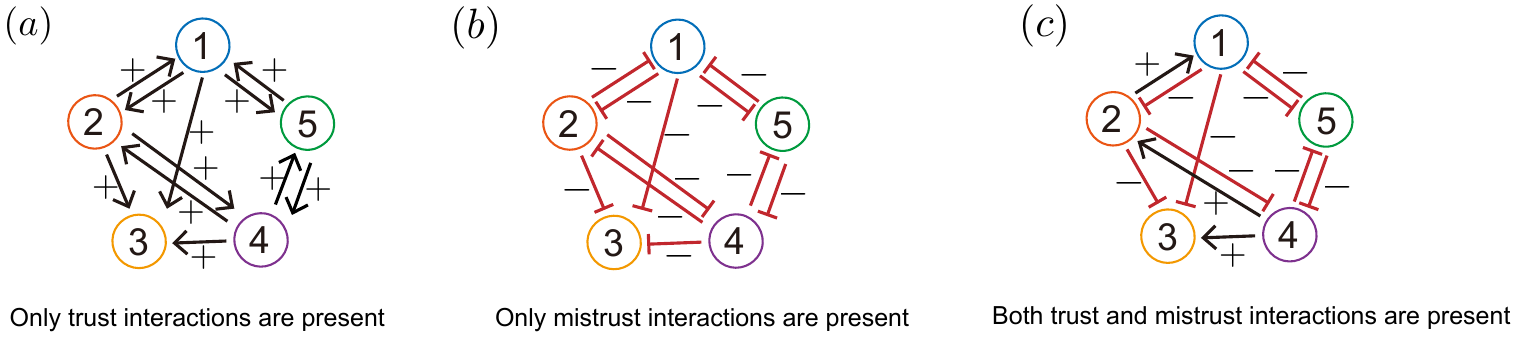}
    \caption{Three interaction networks without self-loops: $\mathcal{G}(W_{t})$, $\mathcal{G}(W_{m})$, and $\mathcal{G}(W_{tm})$. Black lines and red lines represent trust and mistrust interactions, respectively. In $(a)$, only trust interactions are present, while in $(b)$, only mistrust interactions are present. In $(c)$, both trust and mistrust interactions coexist.}
    \label{pic-21}
 \end{figure*}
In Theorem \ref{the1} and Corollary \ref{cor2}, we do not require the assumption that each node has a self-loop. Thus, Theorem \ref{the1} can be seen as an extension of existing research results on both delayed and delay-free systems. In fact, the presence of self-loops in the signed interaction network $\mathcal{G}(W)$ is not a necessary condition for the convergence of discrete-time systems (see Example \ref{exa2}).
 \begin{example}\label{exa2}(Convergence of the discrete-time system (\ref{equ2}) without self-loop)
Consider the system (\ref{equ2}) associated with three interaction networks $\mathcal{G}(W_{t})$, $\mathcal{G}(W_{m})$, and $\mathcal{G}(W_{tm})$, where matrices $W_{tm}$ is given as follows
$$W_{tm}=\begin{bmatrix}\begin{smallmatrix}
0 & 0.3& 0 & 0 & -0.7\\
-0.5 &0& 0 & 0.5 &0\\
-0.5 & -0.3 & 0 &0.2& 0\\
0 & -0.5& 0 & 0 & -0.5\\
-0.5 & 0& 0 & -0.5 & 0\\
\end{smallmatrix}\end{bmatrix}.$$
Moreover, $W_{t}=|W_{tm}|$ and  $W_{m}=-|W_{tm}|$.
In Fig. \ref{pic-21} $(a)$ and $(b)$, only trust or mistrust interactions are present, while in Fig. \ref{pic-21} $(c)$, trust and mistrust interactions coexist. With the initial opinion $X(0)=[-\frac{1}{2}, -\frac{1}{4}, 0, \frac{1}{3}, \frac{1}{2}]^{\top}$, we perform simulations and depict the test results in Fig. \ref{pic-2}. As shown in Fig. \ref{pic-2}, the discrete-time system (\ref{equ2}) with only trust or mistrust interactions can not converge, whereas the discrete-time system (\ref{equ2}) with coexistence of trust and mistrust interactions converges. It suggests that the self-loop could not be necessary for the convergence of system (\ref{equ2}) when both trust and mistrust interactions are present.
\begin{figure*}[ht]
    \centering
    \includegraphics[scale=0.55]{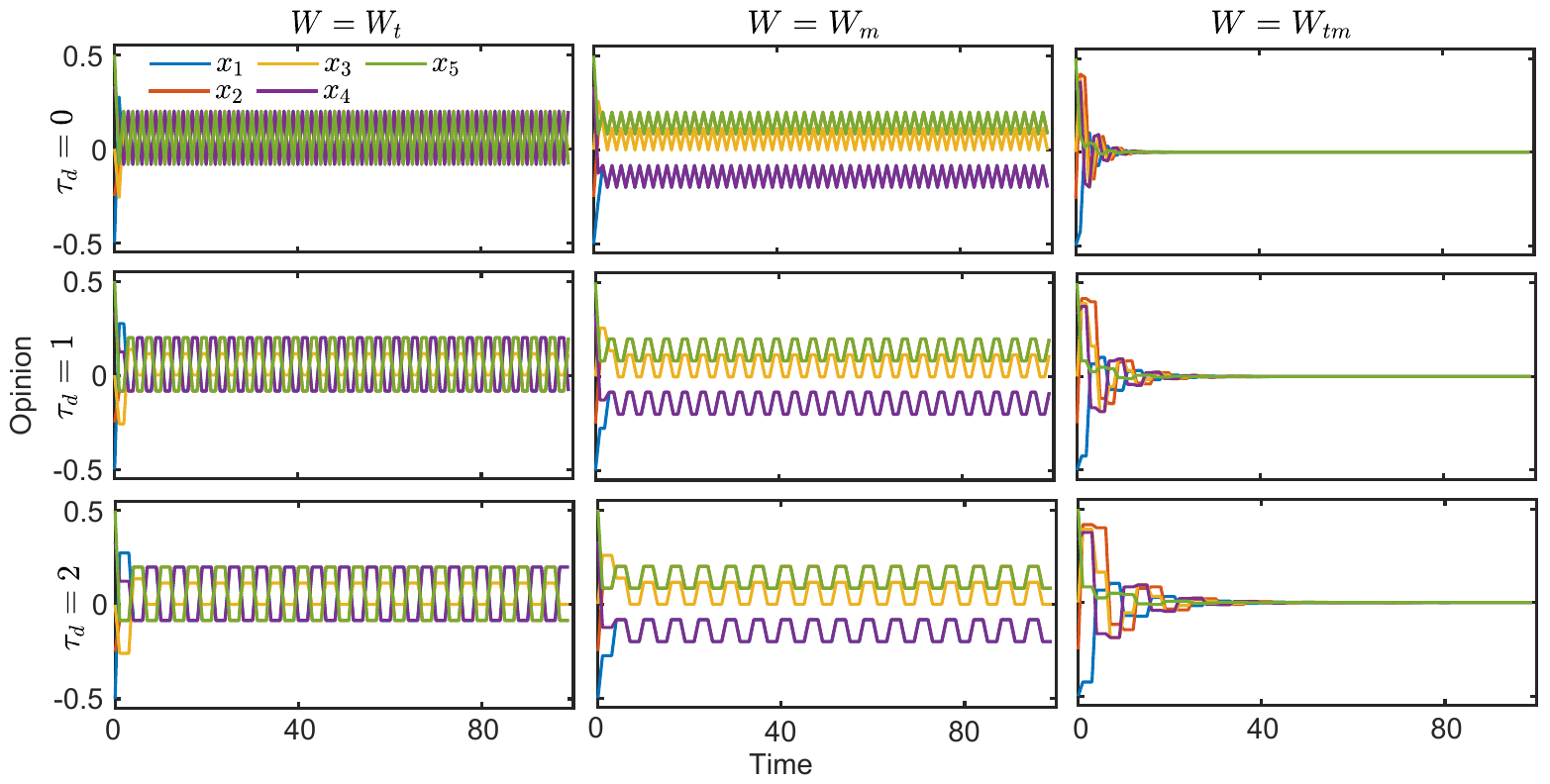}
    \caption{Convergence performances with evolution along the time axis for the discrete-time system (\ref{equ2})
associated with three interaction networks $W=W_{t}$, $W=W_{m}$, $W=W_{tm}$ in Fig. \ref{pic-21}, respectively.}
    \label{pic-2}
 \end{figure*}
\end{example}
Example \ref{exa2} demonstrates that the discrete-time system (\ref{equ2}) with the coexistence of trust and mistrust interactions can converge even if $w_{ii}=0$, i.e., the interaction topology is free of self-loops. This observation naturally leads to the question of whether there exists a class of interaction topologies without self-loops that can still ensure the convergence of system (\ref{equ2}). Subsequently, we address this question by constructing the following signed networks.
\\
\noindent
\textbf{Case 1 (Random mixed interactions):}  For random mixture interactions of various interaction types, we construct the interaction network with $n$ individuals in the following way:
\\
i) individual $j$ initiates an interaction to $i$ with probability $P>0$, i.e., the edge pointing from node $j$ to node $i$ exists with probability $P$ in interaction networks.
\\
ii) if individual $j$ has effect on individual $i$, then we draw the signed interaction strength $s_{ij}$ as random variable $Z$ with expectation 0 and variance $\sigma^2$.
\\
\noindent
 \textbf{Case 2 (Complex mixed interactions):} For the mixture interactions under a certain proportion of five typical interaction types, we construct the interaction network with $n$ individuals in the following way:
 \\
 i) individuals $i$ and $j$ interact with probability $P>0$;
 \\
 ii) the interaction strengths are categorized into five typical scenarios:
 \\
 (1) Mutual trust $(+/+)$ interaction with proportion $P_{+/+}$. The interaction strengths $s_{ij}$ and $s_{ji}$ take the values of $|Z|$  independently.
 \\
 (2) Mutual mistrust $(-/-)$ interaction  with proportion $P_{-/-}$. The interaction strengths $s_{ij}$ and $s_{ji}$ take the values of $-|Z|$  independently.
 \\
 (3) Trust$/$mistrust $(+/-)$ interaction with proportion $P_{+/-}$. The interaction strengths $s_{ij}$ and $s_{ji}$ have opposite signs: one takes the value of $|Z|$ while the other takes the value of  $-|Z|$.
 \\
 (4) Unilateral trust $(+/0)$ interaction with proportion $P_{+/0}$. One of the interaction strengths $s_{ij}$ and $s_{ji}$ takes the value of $|Z|$ while the other takes the value of 0.
 \\
 (5) Unilateral mistrust $(-/0)$ interaction with proportion $P_{-/0}$. One of the interaction strengths $s_{ij}$ and $s_{ji}$ takes the value of $-|Z|$ while the other takes the value of 0.

For these two scenarios, $s_{i j}<0$ ($s_{i j}>0$) represents the mistrust/trust of individual $j$ for individual $i$, and $s_{i j}=0$ denotes that the interaction strength from individual $j$ to individual $i$ is zero. Moreover, we assume that there are no self-loops, and for clarity, we refer to the networks described in Case 1 and Case 2 as the random mixture network and the complex mixture network, respectively

For the random mixture scenario, we have some statistics of the interaction matrix $S=[s_{ij}]$: $$\left\{\begin{array}{l}
\mathbb{E}\left(s_{ij}\right)=P\mathbb{E}\left(Z\right)=0, \\
\mathbb{E}\left(|s_{i j}|\right)=P\mathbb{E}\left(|Z|\right),\\
\mathbb{E}\left(s_{ij}^{2}\right)=P\mathbb{E}\left(Z^{2}\right)=P\sigma^2, \\
\operatorname{Var}\left(s_{i j}\right)=\mathbb{E}\left(s_{ij}^{2}\right)-\mathbb{E}^{2}\left(s_{ij}\right)=P\sigma^2.
\end{array}\right.$$
For large $n$, since $|s_{ij}|$ are i.i.d. from the distribution of $|Z|$, the $i$th row sum $|S|_{i}$ of matrix $|S|$ approaches a constant:
\begin{align}\sum\limits_{j=1}^n |s_{i j}| &\approx(n-1) \mathbb{E}\left(|s_{i j}|\right)\nonumber\\&=(n-1)P\mathbb{E}(|Z|).\label{equ_30}\end{align}

Similarly, for the complex mixture of diverse interaction type, we have some statistics of the interaction matrix $S$
 \begin{equation*}\left\{\begin{array}{l}
\mathbb{E}\left(s_{ij}\right)=P\bar{P}\mathbb{E}\left(|Z|\right), \\
\mathbb{E}\left(|s_{i j}|\right)=P\hat{P}\mathbb{E}\left(|Z|\right),\\
\mathbb{E}\left(s_{ij}^{2}\right)=\displaystyle P\hat{P}\mathbb{E}\left(|Z|^{2}\right)=P\hat{P}\sigma^2, \\
\mathbb{E}\left(s_{i j}s_{j i}\right)=PP^{*}\mathbb{E}^{2}\left(|Z|\right),
\end{array}\right. \end{equation*}
where \begin{equation}\label{equ_28}\hat{P}=P_{+/+}+P_{+/-}+P_{-/-}+\frac{1}{2}P_{+/0}+\frac{1}{2}P_{-/0},\end{equation} \begin{equation}\label{equ_29}\bar{P}=P_{+/+}-P_{-/-}+\frac{1}{2}P_{+/0}-\frac{1}{2}P_{-/0},\end{equation}
and \begin{equation}\label{equ_299}P^{*}=P_{+/+}+P_{-/-}-P_{+/-},\end{equation}
where $0\leq\hat{P}\leq1$, $-1\leq\bar{P}\leq1$ and $\bar{P}\leq\hat{P}$.
For large $n$, the $i$th row sum $|S|_{i}$ of matrix $|S|$ approaches to a constant
\begin{align}\sum\limits_{j=1}^n |s_{i j}|&\approx(n-1) \mathbb{E}\left(|s_{i j}|\right)\nonumber\\ &=(n-1)P\hat{P}\mathbb{E}(|Z|).\label{equ_31}\end{align}
\indent
For both random mixture and complex interaction scenarios, we conclude that system (\ref{equ2}) can be rewritten as
\begin{equation}\label{equ11}X(k+1)=WX(k-\tau_{d}),\end{equation} where $w_{i j}=\displaystyle\frac{s_{ij}}{|S|_{i}}$, if $i\neq j$; $w_{i j}=0$, otherwise. In order to clarify the $i$th row sum $|S|_{i}$ for random mixture interactions and that for complex mixed interactions, we denote them as $C_{\text{r}}$ and $C_{\text{m}}$, respectively.

In what follows, several useful lemmas are presented to establish a solid foundation for our subsequent analysis.
\begin{lemma}\label{lem31} (See Theorem 1.10 (Circular law) in \cite{Tao:10}) Let matrix $W$ be the $n\times n$ random matrix whose entries $w_{ij}$ are i.i.d. random variables with mean zero and variance $\frac{1}{n}$. Then, when $n$ goes to infinity, the spectral distribution of $W$ converges (both in probability and the almost sure sense) to the uniform distribution on the unit disk. \end{lemma}
\begin{lemma}(See \cite{All:12})\label{lem41} Let matrix $W$ be the $n\times n$ random matrix whose entries $w_{ij}$ are random variables with mean zero and variance $\frac{1}{n}$. The asymmetric entries of random matrix $W$ are i.i.d, and symmetric entries obey mean $\frac{z}{n}$. Then, when $n \rightarrow \infty$, the spectral distribution of $W$ converges to the uniform distribution on the complex plane centered at the origin, whose horizontal half-axis length is $1+z$ and the vertical half-axis length is $1-z$, i.e.,
$${\left(\frac{x}{1+z}\right)}^{2}+{\left(\frac{y}{1-z}\right)}^{2}\leq 1.$$
 \end{lemma}
In this paper, we consider the general case where interaction network contains both trust and mistrust relationships between agents, namely, $P_{+/+}+P_{+/0}<1$ and $P_{-/-}+P_{-/0}<1$. Moreover, it should be pointed out that the convergence in this paper is considered for large system size $n$.
\begin{theorem}\label{the3}
The discrete-time system (\ref{equ2}) with random or complex mixed interactions converges to zero.
\begin{proof}
We prove this theorem in the following three steps.
\\
\textbf{Step 1:}  In this step, we focus on the eigenvalues distribution of $W$ in the random mixture case.

By the definition of matrix $W$ in Eq. (\ref{equ11}), we have
$$\left\{\begin{array}{l}
\mathbb{E}\left(w_{ij}\right)=\displaystyle\frac{\mathbb{E}\left(s_{ij}\right)}{C_{\text{r}}}=0, \\
\mathbb{E}\left(w_{ij}^{2}\right)=\displaystyle\frac{\mathbb{E}\left(s_{ij}^{2}\right)}{C_{\text{r}}^2}=\displaystyle\frac{P\sigma^2}{C_{\text{r}}^2}, \\
\operatorname{Var}\left(w_{i j}\right)=\mathbb{E}\left(w_{ij}^{2}\right)-\mathbb{E}^{2}\left(w_{ij}\right)=\mathbb{E}\left(w_{ij}^{2}\right).
\end{array}\right.$$
Let $R=[r_{ij}]=\displaystyle\frac{W}{\sqrt{n\operatorname{Var}\left(w_{i j}\right)}}$, then we have
$$\left\{\begin{array}{l}
\mathbb{E}\left(r_{ij}\right)=0, \\
\mathbb{E}\left(r_{ij}^{2}\right)=\displaystyle\frac{1}{n}, \\
\operatorname{Var}\left(r_{i j}\right)=\mathbb{E}\left(r_{ij}^{2}\right)-\mathbb{E}^{2}\left(r_{ij}\right)=\displaystyle\frac{1}{n}.
\end{array}\right.$$
 According to Lemma \ref{lem31}, the eigenvalues of $R$ are uniformly distributed on a unit circle centered at $(0,0)$, as $n \rightarrow \infty$. It follows that when $n$ is sufficiently large, the eigenvalue distribution of $W$ is uniform distributed on a circle of radius approximately \begin{align}\rho(W)&=\sqrt{n\operatorname{Var}\left(w_{i j}\right)}=\displaystyle\frac{\sqrt{nP\sigma^2}}{C_{\text{r}}}\nonumber\\&=\frac{\sqrt{nP\sigma^2}}{(n-1)P\mathbb{E}(|Z|)}.\label{equ_3}\end{align}
Thus, for large $n$, we have $\rho(W)<1$.
\\
\textbf{Step 2:} In this step, we study the eigenvalues distribution of $W$ in complex mixture case.

Just as for random mixture interactions, we first consider the eigenvalue distribution of matrix $W$. Then, we have
 \begin{equation}\label{equ_33}\left\{\begin{array}{l}
\mathbb{E}:=\mathbb{E}\left(w_{ij}\right)=\frac{\mathbb{E}\left(s_{ij}\right)}{C_{\text{m}}}=\frac{P\bar{P}\mathbb{E}\left(|Z|\right)}{C_{\text{m}}},\\
\mathbb{E}\left(w_{ij}^{2}\right)=\frac{\mathbb{E}\left(s_{ij}^{2}\right)}{C_{\text{m}}^2}=\frac{P\hat{P}\sigma^2}{C_{\text{m}}^2},\\
\mathbb{V}:=\operatorname{Var}\left(w_{i j}\right)=\mathbb{E}\left(w_{ij}^{2}\right)-\mathbb{E}^2,\\
\mathbb{T}:=\mathbb{E}\left(w_{ij}w_{ji}\right)=\frac{PP^{*}\mathbb{E}^{2}\left(|Z|\right)}{C_{\text{m}}^2},
\end{array}\right. \end{equation}
where $\hat{P}$, $\bar{P}$ and $P^{*}$ are defined in (\ref{equ_28}), (\ref{equ_29}) and (\ref{equ_299}), respectively.
\\
\indent
Let $N=[n_{ij}]=\bar{W}-\mathbb{E} \cdot \textbf{1}\cdot \textbf{1}^{\top}+\mathbb{E}\cdot I$, then we can obtain some statistics of matrix $N$. Specifically,
$$\left\{\begin{array}{l}
\mathbb{E}\left(n_{ij}\right)=0,\\
\mathbb{E}\left(n_{ij}^{2}\right)=\mathbb{V}=\operatorname{Var}\left(n_{i j}\right),\\
\mathbb{E}\left(n_{ij}n_{ji}\right)=\mathbb{E}\left((w_{ij}-\mathbb{E} )(w_{ji}-\mathbb{E} )\right)=\mathbb{T}-\mathbb{E}^2.
\end{array}\right.$$
\indent
Furthermore, let $V=[v_{ij}]=\displaystyle\frac{N}{\sqrt{n\mathbb{V}}}$, then
\[\left\{\begin{array}{l}
\mathbb{E}\left(v_{ij}\right)=0, \\
\mathbb{E}\left(v_{ij}^{2}\right)=\displaystyle\frac{1}{n}, \\
\mathbb{E}\left(v_{ij}v_{ji}\right)=\displaystyle\frac{\zeta}{n}, \\
\end{array}\right.
\]
where $\zeta=\frac{(\mathbb{T}-\mathbb{E}^2)}{\mathbb{V}}\in\mathbb{R}$. According to Lemma \ref{lem41}, when $n$ is sufficiently large, the eigenvalues of $F$ are uniformly distributed in an ellipse centered at (0, 0)
\[
{\left(\frac{x}{1+\zeta}\right)}^{2}+{\left(\frac{y}{1-\zeta}\right)}^{2}\leq 1.
\] It follows that the eigenvalues of $N$ are uniformly distributed in an ellipse centered at (0, 0) and
\[
{\left(\frac{x}{\sqrt{n\mathbb{V}}(1+\zeta)}\right)}^{2}+{\left(\frac{y}{\sqrt{n\mathbb{V}}(1-\zeta)}\right)}^{2}\leq 1.
\]
\indent
Note that $\mathbb{E} \cdot \textbf{1}\cdot \textbf{1}^{\top}$ is a rank-one perturbation matrix with $n-1$ zero eigenvalues and $n\mathbb{E}$ is a single eigenvalue. According to the low-rank perturbation theorem, when $|n\mathbb{E}|\leq \sqrt{n\mathbb{V}}$, all eigenvalues of $N+\mathbb{E} \cdot \textbf{1}\cdot \textbf{1}^{\top}$  are still uniformly distributed in the ellipse above. Finally, the effect of $\mathbb{E}\cdot I$ is to shift the whole distribution by subtracting
$\mathbb{E}$ from each eigenvalue. When $|n\mathbb{E}|>\sqrt{n\mathbb{V}},$ $n-1$ eigenvalues of $N+\mathbb{E} \cdot \textbf{1}\cdot \textbf{1}^{\top}$ are still uniformly distributed in the ellipse above, whereas an eigenvalue $\hat\lambda$ is modified as
\begin{align*}\hat\lambda&=n\mathbb{E}\left(w_{ij}\right)+\frac{\mathbb{E}\left(n_{ij}n_{ji}\right)}{\mathbb{E}\left(w_{ij}\right)}
=n\mathbb{E}+\frac{\mathbb{T}-\mathbb{E}^2}{\mathbb{E}}.\end{align*}
Therefore, for sufficiently large $n$, we obtain the eigenvalue distribution of $W$.

When $|n\mathbb{E}|\leq \sqrt{n\mathbb{V}}$, the eigenvalues of $W$  are uniformly distributed in the ellipse
\begin{equation}\label{equ42}{\left(\frac{x+\mathbb{E}}{a_{N}}\right)}^{2}+{\left(\frac{y}{b_{N}}\right)}^{2}\leq 1,\end{equation}
where $a_{N}=\sqrt{n\mathbb{V}}(1+\zeta)$ and $b_{N}=\sqrt{n\mathbb{V}}(1-\zeta)$.

When $|n\mathbb{E}|>\sqrt{n\mathbb{V}}$, there is also an eigenvalue distributed outside the ellipse
\begin{equation}\label{equ43}\left\{\begin{array}{l}
\displaystyle {\left(\frac{x+\mathbb{E}}{a_{N}}\right)}^{2}+{\left(\frac{y}{b_{N}}\right)}^{2}\leq 1,\\
\displaystyle \lambda_{\mathrm{outlier}}=\hat{\lambda}-\mathbb{E}.
\end{array}\right.\end{equation}
\indent
 Three endpoints $Q_{\mathrm{rightmost}}$, $Q_{\mathrm{leftmost}}$ and $Q_{\mathrm{uppermost}}$ (rightmost, leftmost, and uppermost) of the distribution in above ellipse and the unique endpoint $Q_{\mathrm{outlier}}$ corresponding to eigenvalue $\lambda_{\mathrm{outlier}}$ outside the ellipse (if it exists) can then be estimated as
\[
\left\{\begin{array}{l}
Q_{\mathrm{rightmost}}=(a_{N}-\mathbb{E}, 0),\\
Q_{\mathrm{leftmost}}=(-a_{N}-\mathbb{E}, 0), \\
Q_{\mathrm{uppermost}}=(-\mathbb{E}, b_{N}),\\
Q_{\mathrm{outlier}}=(\lambda_{\mathrm{outlier}}, 0)=(\hat{\lambda}-\mathbb{E}, 0).
\end{array}\right.
\]
Furthermore,
\begin{equation}\label{equ_12}\left\{\begin{array}{l}
\mathbb{E}=\frac{\mathbb{E}\left(s_{ij}\right)}{C_{\text{m}}}=\frac{\bar{P}}{(n-1)\hat{P}}, \\
w_{ii}-\mathbb{E}=\frac{-\bar{P}}{(n-1)\hat{P}}\\
\sqrt{n\mathbb{V}}=\frac{\sqrt{nP\hat{P}\sigma^2-nP^2\bar{P}^2\mathbb{E}^2\left(|Z|\right)}}{(n-1)P\hat{P}\mathbb{E}\left(|Z|\right)},\\
\zeta=\frac{P^{*}\mathbb{E}^2\left(|Z|\right)-P\bar{P}^2\mathbb{E}^2\left(|Z|\right)}{\hat{P}\sigma^2-P\bar{P}^2\mathbb{E}^2\left(|Z|\right)},\\
\lambda_{\mathrm{outlier}}=\frac{P^{*}+(n-2)P\bar{P}^2}{(n-1)P\hat{P}\bar{P}}.
\end{array}\right.\end{equation}
\indent
According to (\ref{equ_12}), for large $n$, $\sqrt{n\mathbb{V}}>0$ is sufficiently small and $\zeta$ is bound. Moreover, $a_{N}$ and $b_{N}$ are sufficiently small. Furthermore, $|\mathbb{E}|<1$. Moreover, for large $n$, we have
\begin{align*}\lambda_{\mathrm{outlier}}&\approx\frac{(n-2)P\bar{P}^2
}{(n-1)P\hat{P}\bar{P}}
\approx\frac{\bar{P}}{\hat{P}}\label{equ_34}.\end{align*}
Hence, $|\lambda_{\mathrm{outlier}}|\approx 1$ if and only if $|\bar{P}|=\hat{P}$. Since $P_{+/+}+P_{+/0}<1$ and $P_{-/-}+P_{-/0}<1$, we have $|\lambda_{\mathrm{outlier}}|<1$. In summary, for large $n$, $|Q_{\mathrm{rightmost}}|<1$, $|Q_{\mathrm{rightmost}}|<1$, $|Q_{\mathrm{uppermost}}|<1$ and $|\lambda_{\mathrm{outlier}}|<1$. By employing the eigenvalue estimation method concerning four points ($Q_{\mathrm{rightmost}}$, $Q_{\mathrm{leftmost}}$, $Q_{\mathrm{uppermost}}$,  and $Q_{\mathrm{outlier}}$), we have $\rho(W)<1$.
\\
\textbf{Step 3:} In this step, we establish the connection between the eigenvalues of $W$ and that of $A$ in the random mixture case and complex mixed case.

Since
\begin{equation}\label{equ12}\mathrm{det}(\lambda I-A)=\mathrm{det}(\lambda^{\tau_d+1}I-W)=0,\end{equation}
we have $\theta^{\tau_d+1}-\lambda_{i}=0,$ where $\theta$ and $\lambda_{i}$ are the eigenvalues of matrices $A$ and $\tilde{W}$, respectively.

When time delay is considered, Eq. (\ref{equ12}) has $\tau_d+1$
roots for $\tilde{\lambda}_{i}$ and the particular solution takes the form $\theta^{\tau_d+1}-\lambda_{i}=0$. Then, we obtain   $$|\theta^{\tau_d+1}|=|\theta|^{\tau_d+1}=|\lambda_{i}|<1,$$ and $$|\theta|=|\lambda_{i}|^{\frac{1}{\tau_d+1}}<1.$$
Summarizing the above analysis, we complete the proof of this theorem.
\end{proof}
\end{theorem}
\begin{remark}
Compared to existing research \cite{XI08,Liu19}, Theorem \ref{the1} relaxes the conditions for convergence by shifting from the requirement that all nodes have self-loops to the requirement that every CSCC contains self-loops. Furthermore, Theorem \ref{the3} provides some interaction networks without self-loops in any individual node, which theoretically indicates that the presence of self-loops is not a necessary condition for ensuring convergence.
\end{remark}
\subsection{Convergence rate of the discrete-time opinion dynamics}\label{subsec3.3}
As shown in Fig. \ref{pic-2}, time delay has a significant effect on the convergence rate of discrete-time opinion dynamics (\ref{equ2}), which motivates us to further theoretically analyze the specific effect of time delay on the convergence rate of opinion dynamics (\ref{equ2}).
 \begin{definition}
The convergence rate of system (\ref{equ2}) is defined as $R_{\tau_d}=-log|\theta_{1}|$, where
 the eigenvalues $\theta_{i}:=\lambda_{i}(A)$ are ordered according to \begin{equation}\label{equ131}|\theta_{1}|\geq|\theta_{2}|
\geq\cdots\geq|\theta_{n(\tau_{d}+1)}|.\end{equation}
\end{definition}
\begin{theorem}\label{the4}
For both random and complex mixed interaction scenarios, the convergence rate of  system (\ref{equ2}) decreases as the time delay $\tau_{d}$ increases.
\begin{proof}
Since $w_{ii}=0$, $\forall i\in[n]$, we have $\hat{W}=O$ and $\tilde{W}=W$. Then, $\theta$ is the root of the characteristic polynomial $$\theta^{\tau_d+1}-\lambda_{i}=0, i\in[n],$$  where $\lambda_{i}$ and $\theta$ are the $i$th eigenvalue of matrix $W$ and the eigenvalue of matrix $A$, respectively. Since $|\lambda_{i}|<1$, we obtain the function $f(|\lambda_{i}|,\tau_d):=|\lambda_{i}|^{\frac{1}{\tau_d+1}}$, which is monotonically decreasing with respect to $\tau_d$ and monotonically increasing with respect to $|\lambda_{i}|$.
\end{proof}
\end{theorem}
\begin{remark}
The analysis of the random and complex interaction scenarios not only reveals that the system can still converge even without self-loops, but also facilitates our investigation into the impact of time delays on the convergence rate of discrete-time system (\ref{equ2}).
\end{remark}
Based on the monotonicity of  $f(|\lambda_{i}|,\tau_d):=|\lambda_{i}|^{\frac{1}{\tau_d+1}}$ with respect to $|\lambda_{i}|$, we obtain that the larger the convergence rate of delay-free system is, the larger the convergence rate of corresponding delayed system is. Accordingly, the findings from our previous work in \cite{Yao} regarding the impact of interaction types on the convergence rate of delay-free systems are equally applicable to the delayed system (\ref{equ2}).
\subsection{Convergence of the continuous-time opinion dynamics}
In this subsection, we focus on the convergence boundary of time delay to ensure the convergence of the continuous-time system (\ref{equ1c1}) with  random and complex mixture networks.
\begin{example}\label{exa4}(Convergence of the continuous-time system (\ref{equ1c1}))
Consider the continuous-time system (\ref{equ1c1}) with three signed interaction matrices $W=W_{t}$, $W=W_{m}$, and $W=W_{tm}$ given in Example \ref{exa2}. From Fig. \ref{fig09}, the system (\ref{equ1c1}) without delay and with delay $\tau_{c}=0.2$ converges, but diverges when $\tau_{c}=1$.
\begin{figure*}[ht]
  \centering
    \includegraphics[scale=0.47]{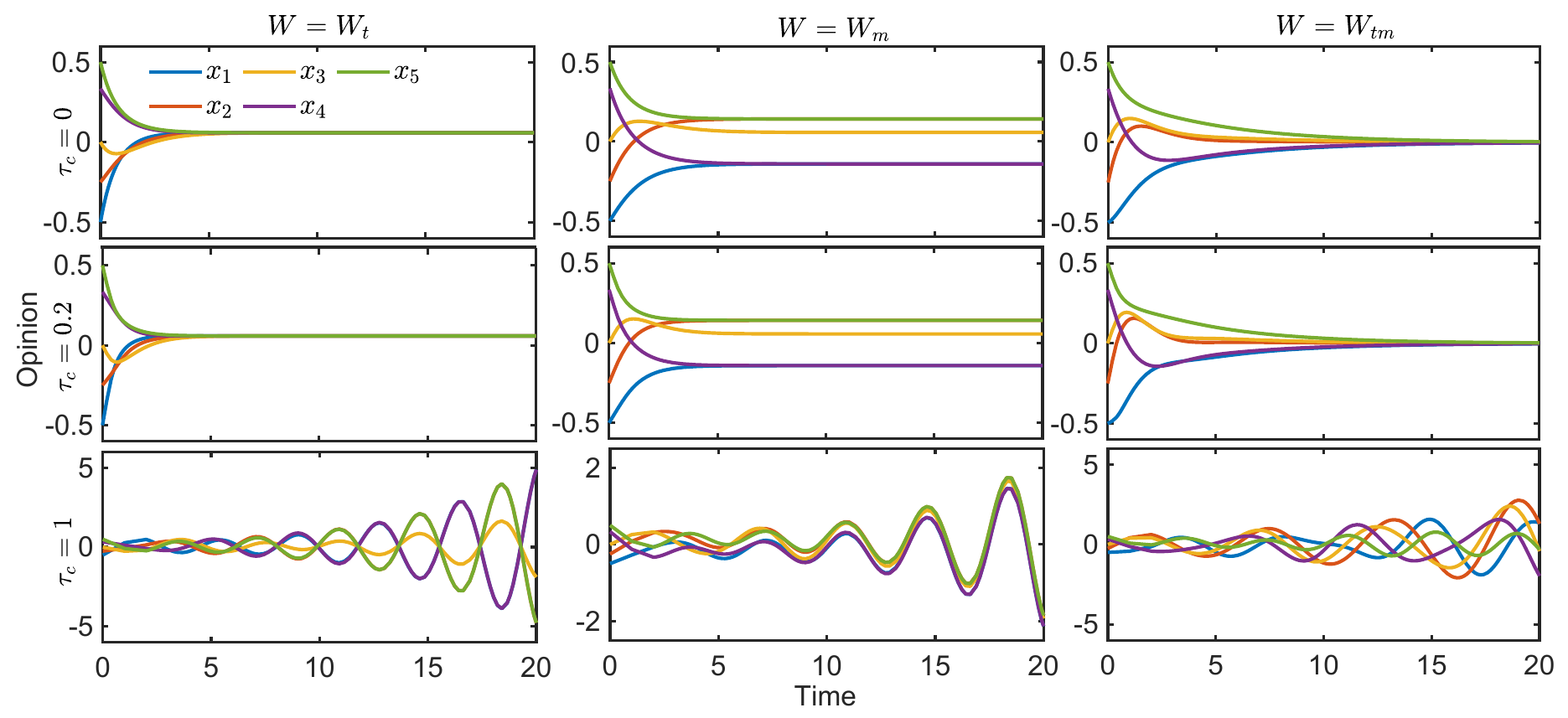}
   \caption{Convergence performances with evolution along the time axis for the continuous-time system (\ref{equ1c1}) associated with three interaction networks of Fig. 2, respectively. }
   \label{fig09}
   \end{figure*}
\end{example}
Example \ref{exa4} demonstrates that, unlike the discrete-time system (\ref{equ2}), the continuous-time system (\ref{equ1c1}) with delays does not necessarily converge. This observation motivates us to investigate the feasible region of the delay to ensure the convergence of system (\ref{equ1c1}).
\begin{lemma}\label{le5}
If all the eigenvalues lie within the interior of the curve $C$, then the system (\ref{equ1c1}) converges to zero, where
\begin{equation}\label{equ471}C(x,y)=\frac{1}{\sqrt{x^2+y^2}} \arctan\left(-\frac{x}{y}\right),\end{equation} $x\in\{\mathrm {Re}(\alpha_{1}),\mathrm {Re}(\alpha_{2}),\cdots,\mathrm {Re}(\alpha_{n})\}$ and $y\in\{\mathrm {Im}(\alpha_{1}),$ $\mathrm {Im}(\alpha_{2}),\cdots,\mathrm {Im}(\alpha_{n})\}$. Moreover, as $\tau_{c}$ increases, this convergence region shrinks.
\begin{proof}
Since the sign change (from a negative value to a positive value) of $\operatorname{Re}(z)$ occurs when $z=i \cdot \omega, \omega \in \mathbb{R}$, by the symmetry of $z$ given in (\ref{c4}), we only need to consider nonnegative $\omega$, that is, $z=i \cdot \omega, \omega \in \mathbb{R}_0^{+}$. Substituting $z=i \cdot \omega$ into the characteristic Eq. (\ref{c4}), we have
\begin{equation}\label{eqn1}i \cdot \omega-(x+i \cdot y) \mathrm{e}^{-i \omega \tau_{c}}=0,\end{equation}
where $x=\operatorname{Re}(\lambda), y=\operatorname{Im}(\lambda)$. Moreover, Eq. (\ref{eqn1}) can be further written as
\begin{equation}\label{c5}\left\{\begin{array}{l}x \cos (\omega \tau_{c})+y \sin (\omega \tau_{c})=0, \\-x \sin (\omega \tau_{c})+y \cos (\omega \tau_{c})=\omega.
\end{array}\right.\end{equation}
\indent
Firstly, we claim the equality $\omega=\sqrt{x^2+y^2}$ holds.
Square the first equation in Eq. (\ref{c5}):
\begin{equation}\label{c9}(x \cos (\omega \tau_{c})+y \sin (\omega \tau_{c}))^2=0.\end{equation}
Expanding the above Eq. (\ref{c9}) yields:
\begin{equation}\label{c10}x^2 \cos ^2(\omega \tau_{c})+2 x y \cos (\omega \tau_{c}) \sin (\omega \tau_{c})+y^2 \sin ^2(\omega \tau_{c})=0.\end{equation}
Similarly, square the second equation in Eq. (\ref{c5}):
\begin{equation}\label{c11}(-x \sin (\omega \tau_{c})+y \cos (\omega \tau_{c}))^2=\omega^2.\end{equation}
Expanding the above Eq. (\ref{c11}) yields:
\begin{equation}\label{c12}x^2 \sin ^2(\omega \tau_{c})-2 x y \sin (\omega \tau_{c}) \cos (\omega \tau_{c})+y^2 \cos ^2(\omega \tau_{c})=\omega^2.\end{equation}
Then, summing Eqs. (\ref{c10}) and (\ref{c12}),  by $\cos ^2(\omega \tau_{c})+\sin ^2(\omega \tau_{c})=1$, we obtain $$\omega=\sqrt{x^2+y^2}.$$
\indent
Case 1: If $y=0$, then $x\cos (\omega \tau_{c})=0$. Since $x\leq0$, we have $$\omega \tau_{c}=l\pi+\frac{\pi}{2}, l\in\mathbb{Z},$$ or $$\omega \tau_{c}=-l\pi+\frac{\pi}{2}, l\in\mathbb{Z},$$ or $$x=0.$$
\\
Moreover, if $x\neq 0$, since $\omega=\sqrt{x^2+y^2}=-x$ and $\tau_{c}^{*}\geq0, l\in\mathbb{Z}$, we have $$\omega \tau_{c}=l\pi+\frac{\pi}{2}, l\in\mathbb{Z},$$ i.e., $$\tau_{c}=-\frac{(2l+1)\pi}{2x}, l\in\mathbb{Z}.$$  Taking  the periodicity into account, the boundary is $(0, 0)$ or \begin{equation}\label{equ_35}(-\frac{\pi}{2x}, 0).\end{equation}
\\
Case 2: If $y\neq0$, by solving the first equality in Eq. (\ref{c5}), we have
\begin{equation}\label{c6}\cos (\omega \tau_{c})=-\frac{y\sin (\omega \tau_{c})}{x}.\end{equation}
Substituting Eq. (\ref{c6}) into the second equality in Eq. (\ref{c5}), we have
\begin{equation}\label{c7}\sin (\omega \tau_{c})=-\frac{\omega x}{x^2+y^2}.\end{equation}
Next, substituting Eq. (\ref{c7}) into the first equality in Eq. (\ref{c5}) yields
\begin{equation}\label{c8}\cos (\omega \tau_{c})=-\frac{y\omega}{x^2+y^2}.\end{equation}
Since
$$\tan(\omega \tau_{c})= \frac{\sin (\omega \tau_{c})}{\cos (\omega \tau_{c})}=-\frac{x}{y},$$
we have
$$\tau_{c}=\frac{1}{\omega}\arctan(-\frac{x}{y}).$$
Thus, the boundary of $\tau_{c}$ is
$$\frac{1}{\sqrt{x^2+y^2}}\arctan\left(-\frac{x}{y}\right),$$
and this boundary encloses a teardrop-shaped region. Due to the continuous dependence of the roots of Eq. (\ref{c4}), if all eigenvalues of $-L$ are located in this region, then all roots of Eq. (\ref{c4}) other than $z=0$ for $\tau_{c}=0$ are located on the open left-half plane.

If $x=0$ and $y\to 0$, $$\lim_{y\to 0}\frac{1}{\sqrt{x^2+y^2}} arctan\left(-\frac{x}{y}\right)=0.$$
If $x\neq0$, $$\lim_{y\to 0^{+}}\frac{1}{\sqrt{x^2+y^2}}arctan\left(-\frac{x}{y}\right)=-\frac{(2l+1)\pi}{2x},$$ which is the same as Eq. (\ref{equ_35}) in Case 1.

Furthermore, let $x=r\cos\theta$ and $y=r\sin\theta$, then
$$r=\frac{arctan(-\cot\theta)}{\tau_{c}}.$$
Thus, as $\tau_{c}$ increases, $r$ becomes smaller.
In summary, we complete the proof of this theorem.
\end{proof}
\end{lemma}
As shown in Fig. \ref{fig12} and Fig. \ref{fig07}, when $\tau_{c}=1$, since there exist eigenvalues that lie outside the boundary, the system (\ref{equ1c1}) can not achieve convergence. This implies that, unlike the discrete-time system (\ref{equ2}), the introduction of time delays imposes stricter requirements on the eigenvalue distribution. Thus, the system (\ref{equ1c1}) with time delay fails to converge for excessive delays.
\begin{figure*}[ht]
    \centering
    \includegraphics[scale=0.53]{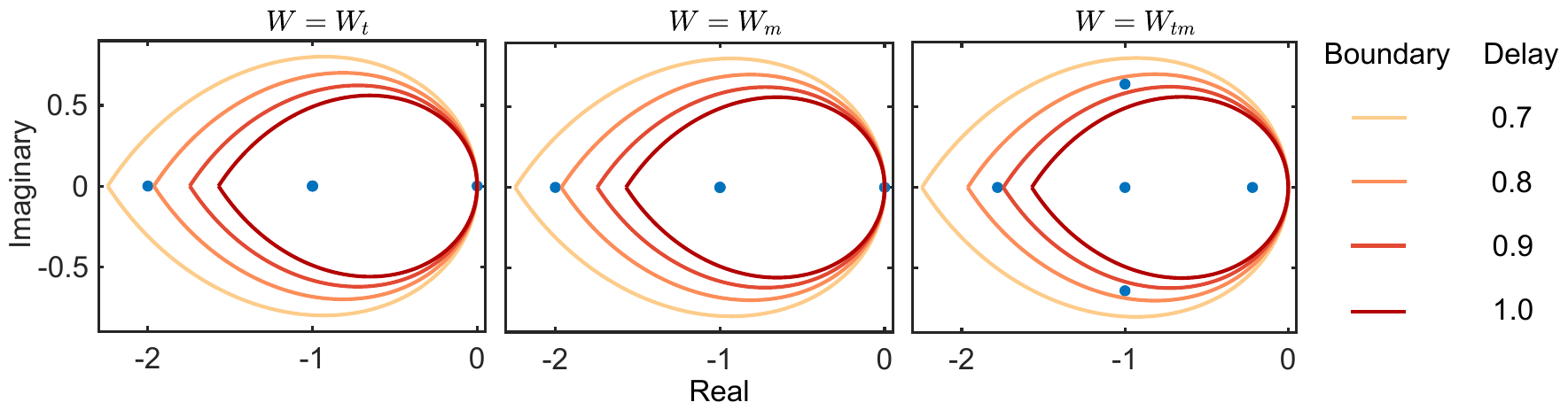}
    \caption{The boundary of the eigenvalue distribution of matrix $-L$ that ensures the convergence of system (\ref{equ1c1}). Solid lines and dots are obtained from our theory results and numerical simulations, respectively. Solid line denotes the curve C in Lemma 7 and dots are the eigenvalues of matrix $-L$.}
    \label{fig12}
    \end{figure*}
\begin{figure*}[ht]
  \centering
  \includegraphics[scale=0.5]{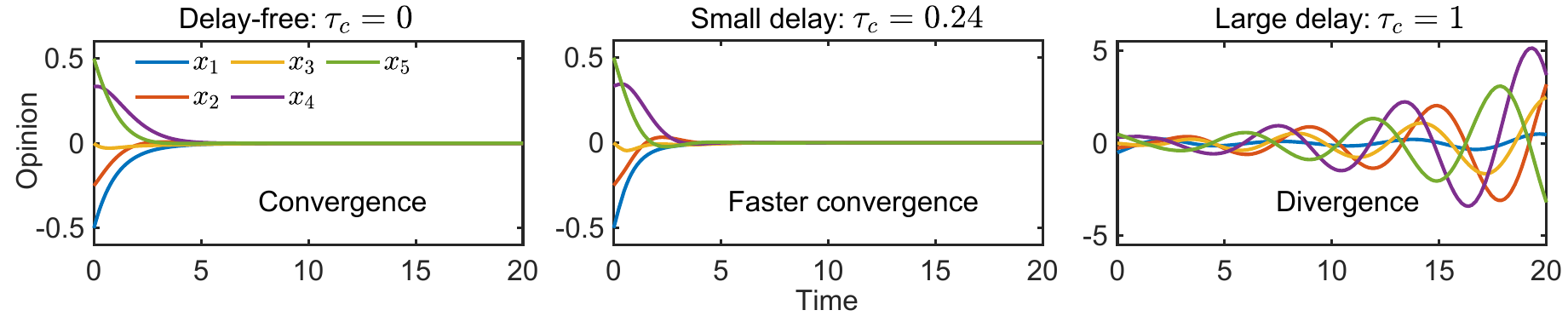}
  \caption{Convergence performances with evolution along the time axis for the continuous-time system (\ref{equ1c1}) with different delays. When $\tau_{c}=0.24$,  the system (\ref{equ1c1}) converges faster than when $\tau_{c}=0$.}
\label{fig07}
\end{figure*}
On the basis of Lemma \ref{le5}, we present the feasible region of the delay ${\tau}_{c}^{*}$ to ensure the convergence of system (\ref{equ1c1}) as follows.
\begin{lemma}\label{lem12}
The continuous-time system (\ref{equ1c1}) is exponentially stable if and only if $\tau_{c}\in[0,{\tau}_{c}^{*})$, where
$${\tau}_{c}^{*}=\min_{i\in[n]}\{{\tau}_{c,i}^{*}\},$$ $${\tau}_{c,i}^{*}=\min_{i\in[n]}\left\{\frac{1}{\sqrt{x_{i}^2+y_{i}^2}} \left|\arctan\left(-\frac{x_{i}}{y_{i}}\right)\right|\right\}.$$
\end{lemma}
\begin{theorem}\label{the8}
 If $\tau_{c}\in[0,{\tau}_{c}^{*})$, the continuous-time system (\ref{equ1c1}) with random mixture interactions converges, where ${\tau}_{c}^{*}=\min\{\tau_{c,u},\tau_{c,l}\}$,
\begin{equation}\label{equ48}\tau_{c,u}=\frac{\alpha}{\sqrt{\alpha^{2}+nP\sigma^2}}
\mathrm{arctan}(\frac{\alpha}{\sqrt{nP\sigma^2}}),\end{equation}
\begin{equation}\label{equ49}\tau_{c,l}=\frac{\alpha \pi}{2\alpha+2\sqrt{nP\sigma^2}},\end{equation}
and
$\alpha=(n-1)P\mathbb{E}(|Z|)$.
\begin{proof}
Since the eigenvalues of matrix $-L$ are obtained by shifting all the eigenvalues of matrix $W$ in the complex plane to the right by one unit, we have three points ($Q_{\mathrm{rightmost}}$, $Q_{\mathrm{leftmost}}$, and $Q_{\mathrm{uppermost}}$):
\begin{align}\label{equ47}
\left\{\begin{array}{l}
Q_{\mathrm{rightmost}}=(-1+\frac{\sqrt{nP\sigma^2}}{(n-1)P\mathbb{E}(|Z|)},0),\\
Q_{\mathrm{leftmost}}=(-1-\frac{\sqrt{nP\sigma^2}}{(n-1)P\mathbb{E}(|Z|)}, 0), \\
Q_{\mathrm{uppermost}}=(-1, \frac{\sqrt{nP\sigma^2}}{(n-1)P\mathbb{E}(|Z|)}).\\
\end{array}\right.
\end{align}
 Notice that the symmetry of the curve $C$ and the distribution of eigenvalues associated with $W$. By the eigenvalue estimation method concerning these three points and Lemmas \ref{le5}-\ref{lem12}, substituting $Q_{\mathrm{rightmost}}$, $Q_{\mathrm{leftmost}}$ and $Q_{\mathrm{uppermost}}$ in Eq. (\ref{equ47}) into Eq. (\ref{equ471}), we then obtain $\tau_{c,r}=\frac{(n-1)\pi P\mathbb{E}(|Z|)}{2(n-1)P\mathbb{E}(|Z|)-2\sqrt{nP\sigma^2}}$, $\tau_{c,u}$ and $\tau_{c,l}$ as given in (\ref{equ48}) and (\ref{equ49}). Since $\tau_{c,r}<\tau_{c,l}$, if $\tau_{c}<\min\{\tau_{c,u},\tau_{c,l}\}$, the continuous-time delayed system (\ref{equ1c1}) with random mixture interactions converges.
\end{proof}
\end{theorem}
\indent
Similarly, for the complex mixture scenario, let
\begin{align}\label{equ50}
\left\{\begin{array}{l}
\tau_{c,r}=\displaystyle\frac{\pi}{2|a_{N}-\mathbb{E}-1|},\\
\tau_{c,l}=\displaystyle\frac{\pi}{2|-a_{N}-\mathbb{E}-1|}, \\
\tau_{c,u}=\displaystyle\frac{1}{\sqrt{(\mathbb{E}-1)^{2}+b_{N}^{2}}}\text{arctan}(\frac{1-\mathbb{E}}{b_{N}}),\\
\tau_{c,\mathrm{out}}=\displaystyle\frac{\pi}{2|\hat\lambda-\mathbb{E}-1|}.
\end{array}\right.
\end{align}
where $a_{N}$, $b_{N}$, $\mathbb{E}$, and $\hat\lambda$ are defined in Theorem \ref{the3}.
\begin{theorem}\label{the9}
 If $\tau_{c}\in[0,{\tau}_{c}^{*})$, the continuous-time system (\ref{equ1c1}) with complex mixed interactions converges, where
\begin{align}\label{equ51}
\left\{\begin{array}{l}
{\tau}_{c}^{*}=\min\{\tau_{c,r},\tau_{c,l},\tau_{c,u}\}, \ |n\mathbb{E}|\leq \sqrt{n\mathbb{V}}, \\
{\tau}_{c}^{*}=\min\{\tau_{c,r},\tau_{c,l},\tau_{c,u},\tau_{c,\mathrm{out}}\}, \ |n\mathbb{E}|>\sqrt{n\mathbb{V}},
\end{array}\right.
\end{align}
where $a_{N}$, $b_{N}$, $\mathbb{E}$, $\mathbb{V}$ and $\hat\lambda$ are defined in Theorem \ref{the3}.
\begin{proof}
Since this theorem can be similarly established as Theorem \ref{the8}, we omit its proof for brevity.
\end{proof}
\end{theorem}
In the sequel, we examine the effect of different interaction types on the convergence of system (\ref{equ1c1}) by considering the following four cases:
\begin{eqnarray}
&&\left\{ \begin{array}{l}
\text{Case}\ (a): \ P_{+/-}=1,P_{+/+}=P_{+/0}=P_{-/-}=P_{-/0}=0, \nonumber \\
\text{Case}\ (b): \ P_{+/-}=P_{+/+}=P_{+/0}=\frac{1}{3},P_{-/-}=P_{-/0}=0, \nonumber \\
\text{Case}\ (c): \ P_{+/-}=P_{-/-}=P_{-/0}=\frac{1}{3},P_{+/+}=P_{+/0}=0, \nonumber \\
\text{Case}\ (d): \ P_{+/+}=P_{+/0}=P_{+/-}=P_{-/-}=P_{-/0}=\frac{1}{5}.
\end{array} \right. \nonumber
\end{eqnarray}
\begin{corollary}\label{cor1}
For Cases $(a), (b), (d)$, when $n$ is sufficiently large, ${\tau}_{c}^{*}\approx\displaystyle\frac{\pi}{2}$. For Case $(c)$, ${\tau}_{c}^{*}\approx\displaystyle\frac{5\pi}{16}$, where ${\tau}_{c}^{*}$ is defined in Lemma \ref{lem12}.
\begin{proof}
By Theorem \ref{the9}, we need to consider $\tau_{c,r}$, $\tau_{c,l}$, $\tau_{c,u}$, and $\tau_{\mathrm{c,out}}$.
For Case $(a)$, we have
$$\hat{P}=1,\bar{P}=0, \text{and}~ P^{*}=-1.$$ Furthermore, we know
\begin{align*}\left\{\begin{array}{l}\mathbb{E}=0,\\ a_{N}=\displaystyle\frac{\sqrt{nP}(\sigma^{2}-\mathbb{E}^{2}(|Z|))}{(n-1)P\sigma\mathbb{E}(|Z|)}, \\ b_{N}=\displaystyle\frac{\sqrt{nP}(\sigma^{2}+\mathbb{E}^{2}(|Z|))}{(n-1)P\sigma\mathbb{E}(|Z|)}, \\ \end{array}\right.\end{align*}
and by Theorem \ref{the9}, we have \begin{align*}\left\{\begin{array}{l}
\tau_{c,r}=\displaystyle\frac{(n-1)\pi P\sigma\mathbb{E}(|Z|)}{|2\sqrt{nP}(\mathbb{E}^{2}(|Z|)-\sigma^{2})-2(n-1)P\sigma\mathbb{E}(|Z|)|},\\
\tau_{c,l}=\displaystyle\frac{(n-1)\pi P\sigma\mathbb{E}(|Z|)}{|2\sqrt{nP}(\sigma^{2}-\mathbb{E}^{2}(|Z|))-2(n-1)P\sigma\mathbb{E}(|Z|)|},\\
\tau_{c,u}=\displaystyle\frac{(n-1) P\sigma\mathbb{E}(|Z|)}{\sqrt{(n-1)^{2}P^{2}\sigma^{2}\mathbb{E}^{2}(|Z|)
+nP(\sigma^{2}+\mathbb{E}^{2}(|Z|))^{2}}}\\
\ \ \ \ \times \mathrm{arctan}\large(\displaystyle\frac{(n-1) P\sigma\mathbb{E}(|Z|)}{\sqrt{nP}(\sigma^{2}+\mathbb{E}^{2}(|Z|))}\large).
 \end{array}\right.\end{align*}
 When $n$ is sufficiently large, we have
$$\tau_{c,r}\approx\tau_{c,l}\approx\tau_{c,u}\approx\frac{\pi}{2}.$$

For Case $(b)$, we have
$$\hat{P}=\frac{5}{6},\bar{P}=\frac{1}{2},P^{*}=0.$$ Moreover, since $|n\mathbb{E}|>\sqrt{n\mathbb{V}}$, there is an eigenvalue distributed outside the ellipse and $\hat{\lambda}\approx \frac{\bar{P}}{\hat{P}}=\frac{3}{5}$.
By virtue of Theorem \ref{the9}, we obtain $\tau_{c,r}\approx\tau_{c,l}\approx\tau_{c,u}\approx\frac{\pi}{2}$ and
$$\tau_{\mathrm{c,out}}\approx\frac{5\pi}{4}>\frac{\pi}{2}.$$

For Case $(b)$, by similar analysis, we know
$$\hat{P}=\frac{5}{6},\bar{P}=-\frac{1}{2},P^{*}=0.$$ Since $|n\mathbb{E}|>\sqrt{n\mathbb{V}}$, there is an eigenvalue distributed outside the ellipse and $\hat{\lambda}\approx \frac{\bar{P}}{\hat{P}}=-\frac{3}{5}$.
By Theorem \ref{the9}, we obtain $\tau_{c,r}\approx\tau_{c,l}\approx\tau_{c,u}\approx\frac{\pi}{2}$ and
$$\tau_{\mathrm{c,out}}\approx\frac{5\pi}{16}<\frac{\pi}{2}.$$

For Case $(d)$, similarly, when $n$ is sufficient large, we have
$$\tau_{c,r}\approx\tau_{c,l}\approx\tau_{c,u}\approx\frac{\pi}{2}.$$

In summary, for Cases $(a), (b), (d)$, the threshold of delay $\tau_{c}^{*}$ that ensures the convergence of system (\ref{equ1c1}) is approximately $\displaystyle\frac{\pi}{2}$. For Case $(c)$, the threshold of delay is approximately $\displaystyle\frac{5\pi}{16}$.
\end{proof}
\end{corollary}
\subsection{Convergence rate of the continuous-time opinion dynamics}
As shown in Fig. \ref{fig07}, there exists a small time delay, such that when the delay is less than this threshold, the convergence rate of system (\ref{equ1c1}) accelerates as the delay increases. Moreover, when $\tau_{c}$ is large, the system fails to converge. This motivates us to
discuss the convergence rate of continuous-time system. In the sequel, we use the Lambert $\mathcal{W}$ function as an effective tool for analyzing delayed systems.

The Lambert $\mathcal{W}$ function is defined as the solution of $w\cdot e^w=z$, $z \in \mathbb{C}$, i.e., $w=\mathcal{W}(z)$. Except for $z=0$ for which $\mathcal{W}(0)=0, W$ is a multi-valued function with an infinite number of solutions denoted by $\mathcal{W}_k(z),$ $k\in\mathbb{Z}$, where $\mathcal{W}_k$ is called the $k$th branch of the $\mathcal{W}$ function. For arbitrary $a\in\mathbb{C}$, $\max\{\text{Re}(\mathcal{W}_{k}(a))|k=0, 1, \cdots, \infty\}=\text{Re}(\mathcal{W}_{0}(a))$ holds \cite{robust}.

Now, we provide the definition of the convergence rate for the continuous-time system (\ref{equ1c1}) with the help of Lambert $\mathcal{W}$ function.
\begin{definition}\label{def1}
 For the continuous-time delayed system (\ref{equ1c}) with delay $\tau_{c}$, if $\mathrm{Re}(\lambda_{i}(L))>0$ for $\forall i\in[n]$, the convergence rate $$R_{\tau_{c}}:=\max_{k\in[n]} \left\{\frac{\mathrm{Re}(W_0(\alpha_k \tau))}{\tau_{c}}\right\},$$  where
 $\alpha_{k}$ is defined in (\ref{equ13}). For the system (\ref{equ1c1}) without delays, the convergence rate $$R_{0}:=|\mathrm{Re}(\alpha_{1})|.$$
\end{definition}
\indent
Next, we adopt the delay rate gain function defined in \cite{Morad20} to analyze $R_{\tau_{c}}$:
$$
g(x)= \begin{cases}\displaystyle\frac{\text{Re}\left(W_0(x)\right)}{\text{Re}(x)}, & x \in \mathbb{C}_{-}, \\ 1, & x=0.\end{cases}
$$
Then, we have
$$R_{\tau_{c}}=\min_{k\in[n]} \left\{g\left(\alpha_k \tau_c\right)\left|\text{Re}\left(\alpha_k\right)\right|\right\}.$$
 To analyze the convergence rate of system (\ref{equ1c1}), we start by giving some notations defined in \cite{Morad20}:
$$\mathcal{I}_{1}:=\left\{k\in\{1,\ldots,n\}|\text{Re}(\alpha_{k})=\text{Re}(\alpha_{1})\right\},$$
$$\mathcal{I}_{\mathrm{in}}:=\left\{k\in\{1,\ldots,n\}|\text{arg}(\alpha_{k})\in\left(\frac{3\pi}{4},\frac{5\pi}{4}\right)\right\},$$
$$\tilde{\tau}_c:=\min_{i\in[n]} \left\{\eta_i\right\},$$ where $\eta_i$ is the unique solution of $g\left(\alpha_i \tau_c\right)= \displaystyle\frac{\left|\text{Re}\left(\alpha_1\right)\right|}{\left|\text{Re}\left(\alpha_i\right)\right|}$ for $\tau_c\in(0,{\tau}_{c,i}^{*})$.
\begin{lemma}\label{lem81}(See Theorem IV.1 in \cite{Morad20})
Consider the linear delayed system $\dot{\mathbf{x}}(t)=\mathbf{A x}(t-\tau)$ and if $\mathrm{Re}(\lambda_{i}(A))<0$ for $\forall i\in[n]$, then there exists $\tau_{c}\in(0,{\tau}_{c}^{*})$ for which $R_{\tau_{c}}>R_{0}$ if and only if $\mathcal{I}_{1}\subset\mathcal{I}_{\mathrm{in}}$.
\end{lemma}
\begin{theorem}\label{the12}
For the continuous-time system (\ref{equ1c1}) with random or complex mixed interactions,
 \\
(1). $R_{\tau_c}>R_0$ for $\tau_{c} \in(0, \tilde{\tau}_{c}) \subset(0, {\tau}_{c}^{*});$
 \\
(2). $R_{\tau_{c}}=R_{0}$ at $\tau_{c}=\tilde{\tau}_{c};$
  \\
(3). $R_{\tau_{c}}<R_{0}$ for $\tau_{c} \in(\tilde{\tau}_{c}, {\tau}_{c}^{*})$. Moreover, $R_{\tau_{c}}$ decreases strictly with $\tau_c \in(\tilde{\tau}_c, {\tau}_{c}^{*})$.
\begin{proof}
From Theorem \ref{the3}, $\mathrm{Re}(\alpha_{k})=\mathrm{Re}(\alpha_{1})$ if and only if $k=1$. Moreover, for the random mixture scenario, $$\alpha_{1}=-\frac{\sqrt{nP\sigma^2}}{(n-1)P\mathbb{E}(|Z|)}-1.$$ For the complex mixture scenario,
\begin{align*}
\alpha_{1}=\left\{\begin{array}{l}
a_{N}-\mathbb{E}-1, \ \text{when} \ |n\mathbb{E}|>\sqrt{n\mathbb{V}},\\
\min\{\hat{\lambda}-\mathbb{E}-1,a_{N}-\mathbb{E}-1\},\ \text{when} \ |n\mathbb{E}|\leq\sqrt{n\mathbb{V}}.\\
\end{array}\right.\end{align*}
Thus, $\text{arg}(\alpha_{1})=\pi$. By Lemma \ref{lem81}, there always exists $\tau_{c}\in(0,{\tau}_{c}^{*})$ such that $r_{\tau_{c}}>r_{0}$, which implies that the convergence rate can increase by small time delay. By Theorem IV. 2 in \cite{Morad20}, this theorem can be proved.
\end{proof}
\end{theorem}
\begin{remark}
Theorem \ref{the12} can be regarded as an extension and expansion of the results presented in \cite{Morad20}. Moreover, our results further show that the diversity of interaction types is beneficial to the convergence rate of time-delayed systems.
\end{remark}
\section{Numerical examples}\label{sect4}
In this section, some numerical examples are given to illustrate our main results.
\begin{example}\label{examm1}(Convergence of the discrete-time system (\ref{equ2})) This example aims to demonstrate the validity of the theoretical results presented in Theorems \ref{the3} and \ref{the4}. In these theorems, the interaction strengths $s_{ij}$ are randomly selected from a normal distribution with mean $\mathbb{E}(Z) = 0$ and variance $\sigma= 1$. Moreover, $n=100$ and $P=0.5$.

(I). At first, we examine the convergence rate of discrete-time system (\ref{equ2}) with random mixture interactions. All eigenvalues of $W$ remain within the unit disk, which implies that results from numerical simulations align well with the theoretical estimations in Theorem \ref{the3}. Furthermore, as the time delay increases, the convergence rate of system (\ref{equ2}) slows down (see Fig. \ref{fig04} for details).
\begin{figure}[ht]
    \centering
    \includegraphics[scale=0.47]{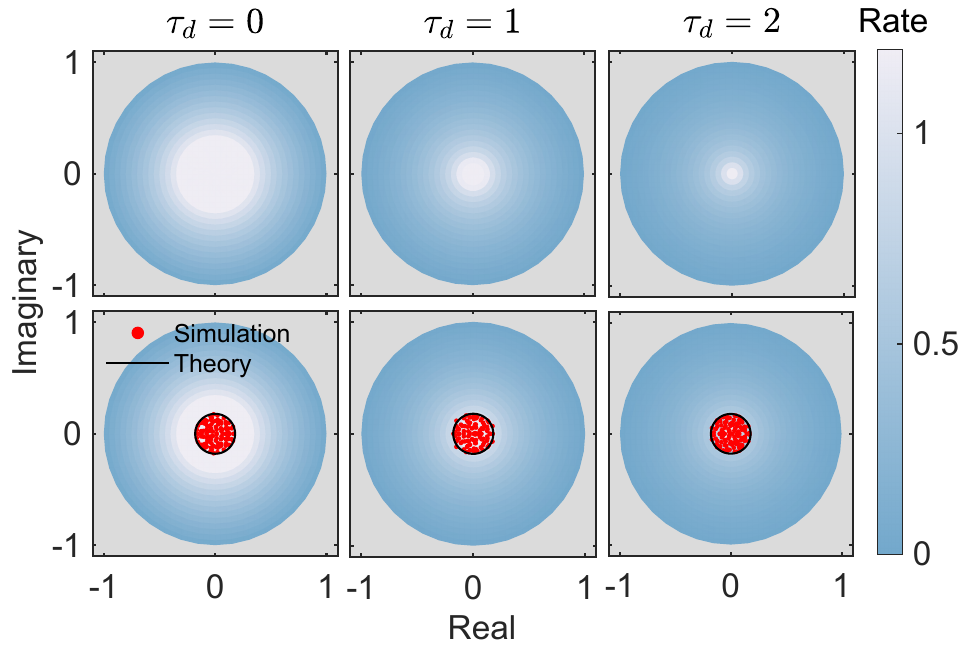}
    \caption{Convergence contour plot in the complex plane at different intensities of time delay. Colors represent the convergence rate, with colors closer to white indicating a larger rate and colors closer to blue indicating a smaller rate. Gray represents the divergence of system (\ref{equ1c1}). Solid lines and dots are obtained from our theory results and numerical simulations, respectively.}
    \label{fig04}
    \end{figure}
\begin{figure}[ht]
    \centering
    \includegraphics[scale=0.52]{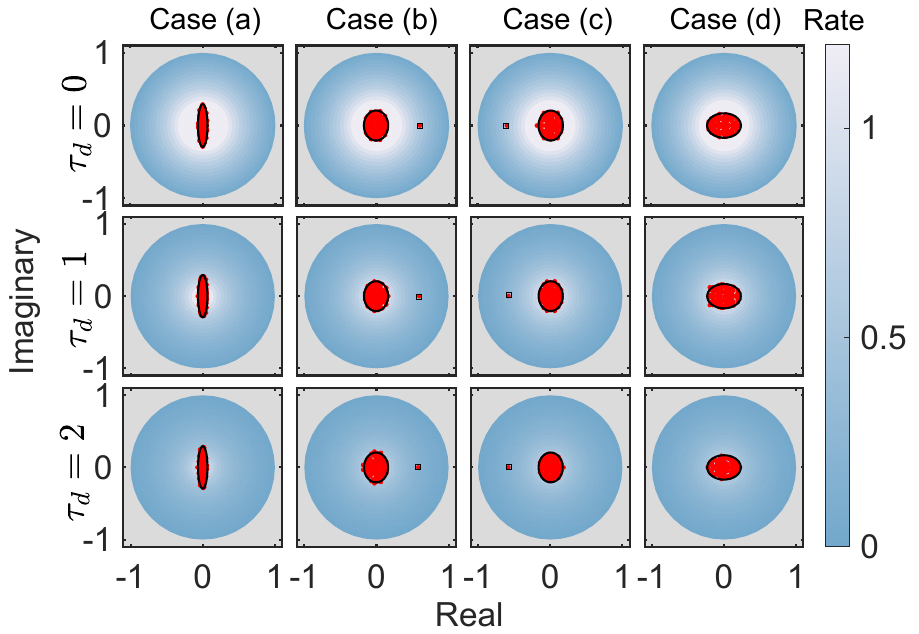}
    \caption{Estimating the eigenvalue distribution of $W$ for four scenarios: \text{Case} $(a)$, \text{Case} $(b)$, \text{Case} $(c)$, and \text{Case} $(d)$.}
    \label{fig05}
    \end{figure}
    \begin{figure}[ht]
    \centering
    \includegraphics[scale=0.5]{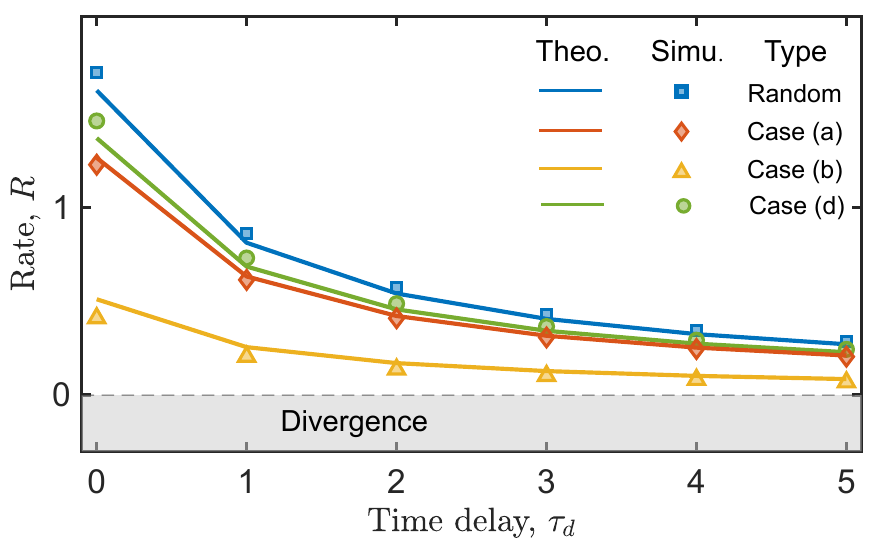}
    \caption{Convergence rate of the system (\ref{equ1}) for four scenarios: random mixture, Case $(a)$, Case $(b)$, and Case $(d)$.}
    \label{fig06}
    \end{figure}

 Next, we investigate the convergence rate of system (\ref{equ2}) with complex mixed interactions by considering four typical cases presented in Corollary 2. As shown in Fig. \ref{fig05}, the distribution of eigenvalues for Case (b) and Case $(c)$ is symmetric with respect to the y-axis, which implies that the convergence rates in these two cases are identical. From Fig. \ref{fig06}, we have $R_{\text{random}} > R_{\text{Case} \ (d)} > R_{\text{Case} \ (a)} > R_{\text{Case} \ (b)}$, indicating that the system (\ref{equ2}) with random mixture interactions exhibits a more rapid convergence than the other cases.
\end{example}

\begin{example}\label{examm3} (Convergence of the continuous-time system (\ref{equ1c1})) In this example, we use the same parameters as in Example \ref{examm1}.
    \begin{figure}[ht]
    \centering
    \includegraphics[scale=0.5]{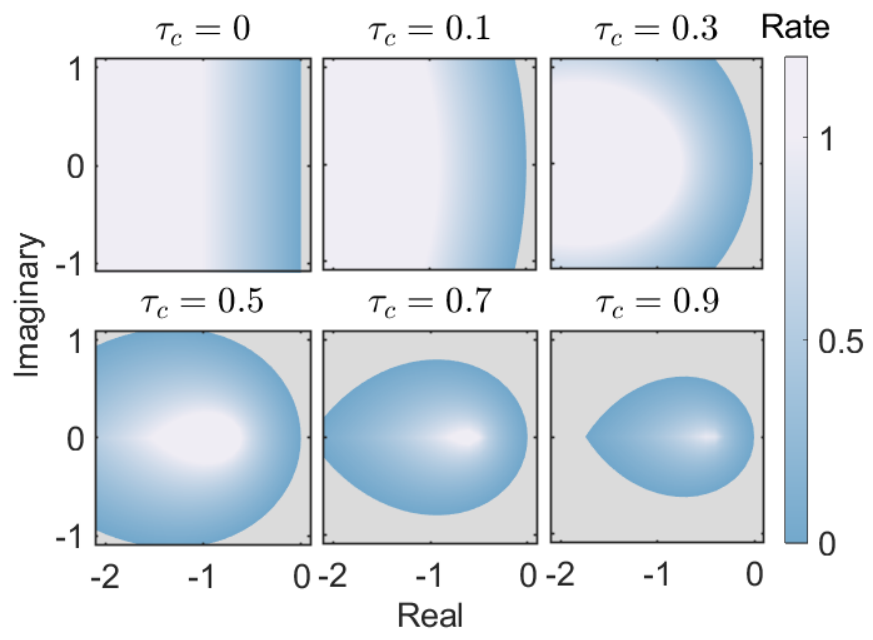}
    \caption{Convergence contour plot in the complex plane at different intensities of time delay. Colors represent the convergence rate of system (\ref{equ1c1}), with colors  closer to white indicating a larger rate and colors  closer to blue indicating a smaller rate. Gray represents the divergence of system (\ref{equ1c1}).}
    \label{fig13}
    \end{figure}
From Fig. \ref{fig13}, it can be observed that the range of feasible eigenvalue distributions for $-L$ that ensures system convergence becomes narrower as the time delay increases. This suggests that large delays will result in the divergence of system (\ref{equ1c1}). Unlike the discrete-time system (\ref{equ2}), the eigenvalues that determine the convergence rate of the continuous-time system (\ref{equ1c1}) are different for Case $(b)$ and Case $(c)$, as shown in Fig. \ref{fig014}.

 As depicted in Fig. \ref{fig15}, for five scenarios (random mixture, \text{Case} $(a)$, \text{Case} $(b)$, \text{Case} $(c)$, \text{Case} $(d)$), the convergence rate of system (\ref{equ1c1}) first increases and then decreases as time delay $\tau_{c}$ increases. Moreover, for the small delays, $R_{\mathrm{Case} \ (a)}>R_{\mathrm{random}}>R_{\mathrm{Case} \ (c)}>R_{\mathrm{Case} \ (d)}>>R_{\mathrm{Case} \ (b)}$. The simulation results are consistent with Theorem \ref{the12}.
 \begin{figure}
    \centering
    \includegraphics[scale=0.48]{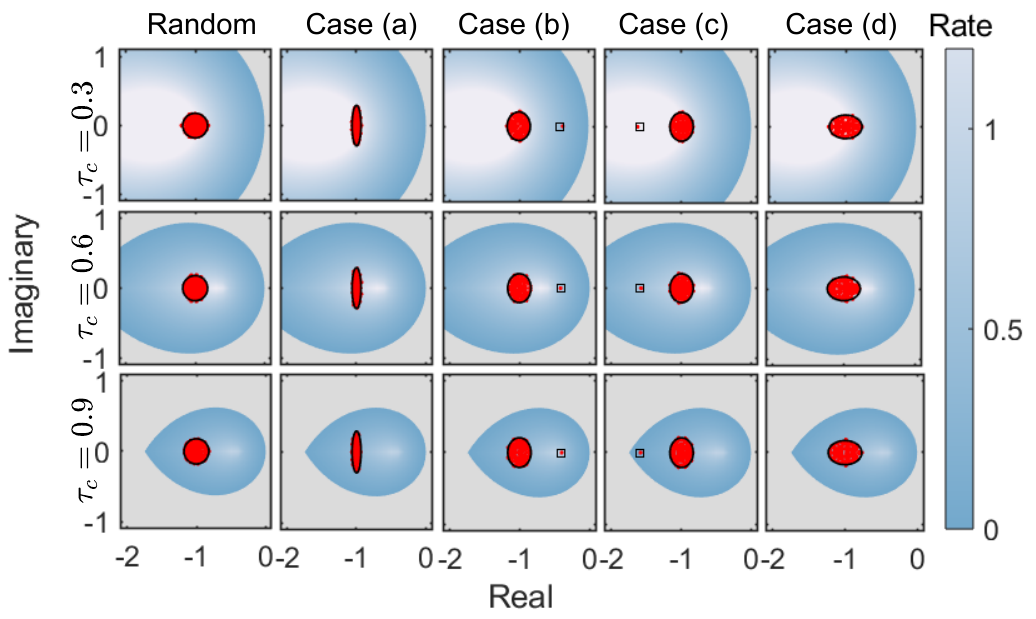}
    \caption{Estimating the eigenvalue distribution of matrix $-L$.}
    \label{fig014}
    \end{figure}
     \begin{figure}
    \centering
    \includegraphics[scale=0.5]{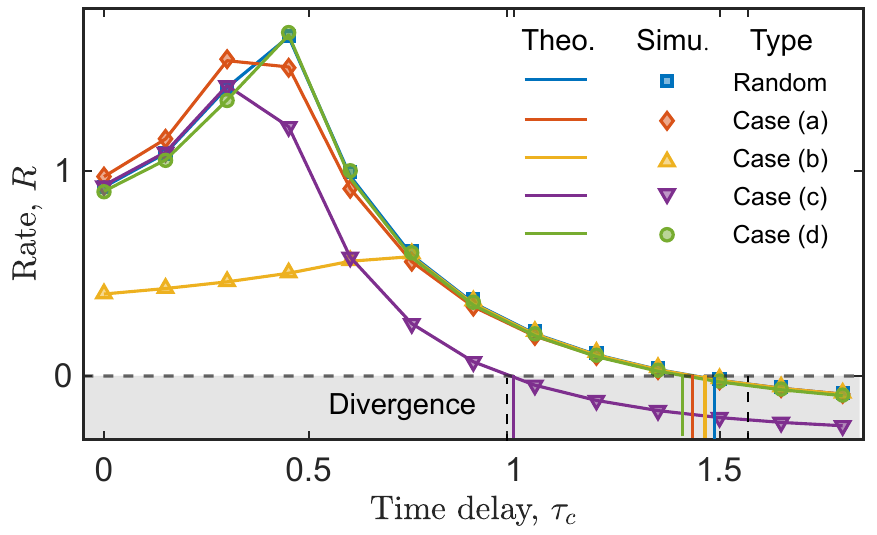}
    \caption{Convergence rate of the system (\ref{equ1c1}) for five scenarios: random mixture, Case $(a)$, Case $(b)$, Case $(c)$, and Case $(d)$. Within the gray box, the intersection points of the vertical solid lines with the x-axis represent the simulated $\tau_{c}^{*}$, while that of the vertical dashed lines with the x-axis represent the $\tau_{c}^{*}$ obtained through theoretical estimation.}
    \label{fig15}
    \end{figure}
\end{example}
\section{Conclusions}\label{sect5}
In this paper, we have systematically established a framework for studying the convergence and convergence rates of both discrete-time and continuous-time opinion dynamics with delays and diverse interaction types. Our results indicate that for continuous-time systems, excessive time delays can lead to system divergence, whereas moderate time delays accelerate convergence, representing a trade-off between time delays and convergence properties. Moreover, different interaction types significantly influence the convergence rate in both delayed and delay-free systems.

Further research in this area could explore the application of our approach to more sophisticated models of opinion dynamics, such as the Friedkin-Johnsen model, the multidimensional opinion model, the model of expressed and private beliefs. In these models, interactions among individuals can be represented as multiplex networks, with time delays occurring both within and between layers. The question of how to adapt our methodology to these complex scenarios remains an open issue for future investigation.

\end{document}